\documentclass[12pt]{amsart}

\usepackage{latexsym}
\usepackage{amssymb}
\usepackage{amsmath}
\usepackage{mathrsfs}

\newtheorem{theorem}{Theorem}[section]
\newtheorem{lemma}[theorem]{Lemma}
\newtheorem{proposition}[theorem]{Proposition}
\newtheorem{corollary}[theorem]{Corollary}

\theoremstyle{definition}
\newtheorem{definition}[theorem]{Definition}

\newtheorem{example}[theorem]{Example}

\theoremstyle{remark}
\newtheorem{remark}[theorem]{Remark}

\def\N{\mathbb{N}}

\begin{document}

\title{Hypernatural numbers as ultrafilters}

\author{Mauro Di Nasso}
\address{Dipartimento di Matematica,
Universit\`{a} di Pisa, Italy.} \email{dinasso@dm.unipi.it}



\begin{abstract}
In this paper we present a use of nonstandard methods 
in the theory of ultrafilters and in related
applications to combinatorics of numbers.
\end{abstract}

\maketitle

\section{Introduction.}

Ultrafilters are really peculiar and multifaced mathematical objects, whose
study turned out a fascinating and often elusive subject. Researchers may
have diverse intuitions about ultrafilters, but they seem to agree on the
subtlety of this concept; \emph{e.g.}, read the following quotations: \emph{%
``The space $\beta\omega$ is a monster having three heads"} (J. van Mill 
\cite{11vm}); \emph{``\ldots the somewhat esoteric, but fascinating and very
useful object $\beta\mathbb{N}$"} (V. Bergelson \cite{11ber}).

The notion of ultrafilter can be formulated in diverse languages of
mathematics: in set theory, ultrafilters are maximal families of sets that
are closed under supersets and intersections; in measure theory, they are
described as $\{0,1\}$-valued finitely additive measures defined on the
family of all subsets of a given space; in algebra, they exactly correspond
to maximal ideals in rings of functions $\mathbb{F}^I$ where
$I$ is a set and $\mathbb{F}$ is a field. 
Ultrafilters and the corresponding construction of ultraproduct are a
common tool in mathematical logic, but they also have many applications in
other fields of mathematics, most notably in topology (the notion of limit
along an ultrafilter, the Stone-\v{C}ech compactification $\beta X$ of a
discrete space $X$, \emph{etc}.), and in Banach spaces (the so-called \emph{%
ultraproduct technique}).

In 1975, F. Galvin and S. Glazer found a beautiful ultrafilter proof of 
\emph{Hindman's theorem}, namely the property that for every finite
partition of the natural numbers $\mathbb{N}=C_1\cup\ldots\cup C_r$, there
exists an infinite set $X$ and a piece $C_i$ such that all sums of distinct
elements from $X$ belong to $C_i$. Since this time, ultrafilters on $\mathbb{%
N}$ have been successfully used also in combinatorial number theory and in
Ramsey theory. The key fact is that the compact space $\beta\mathbb{N}$ of
ultrafilters on $\mathbb{N}$ can be equipped with a pseudo-sum operation, so
that the resulting structure $(\beta\mathbb{N},\oplus)$ is a compact
topological left semigroup. Such a space satisfies really intriguing
properties that have direct applications in the study of structural
properties of sets of integers (See the monography \cite{11hs}, where the
extensive research originated from that approach is surveyed.)

Nonstandard analysis and ultrafilters are intimately connected. In one
direction, ultrapowers are the basic ingredient for the usual constructions
of models of nonstandard analysis since W.A.J. Luxemburg's lecture notes 
\cite{11lux62} of 1962. Actually, by a classic result of H.J. Keisler, the
models of nonstandard analysis are characterized up to isomorphisms as \emph{%
limit ultrapowers}, a class of elementary submodels of ultrapowers which
correspond to direct limits of ultrapowers (see \cite{11ke63} and \cite[\S %
4.4]{11ck}).

In the other direction, the idea that elements of a nonstandard extension $%
{}^*X$ correspond to ultrafilters on $X$ goes back to the golden years of
nonstandard analysis, starting from the seminal paper \cite{11lux69} by
W.A.J. Luxemburg appeared in 1969. This idea was then systematically pursued
by C. Puritz in \cite{11pu1} and by G. Cherlin and J. Hirschfeld in \cite%
{11ch}. In those papers, as well as in Puritz' follow-up \cite{11pu2}, new
results about the Rudin-Keisler ordering were proved by nonstandard methods,
along with new characterizations of special ultrafilters, such as P-points
and selective ultrafilters. (See also \cite{11nr}, where the study of
similar properties as in Puritz' papers was continued.) In \cite{11bl}, A.
Blass pushed that approach further and provided a comprehensive treatment of
ultrafilter properties as reflected by the nonstandard numbers of the
associated ultrapowers.

Several years later, J. Hirschfeld \cite{11hi} showed that hypernatural
numbers can also be used as a convenient tool to investigate certain
Ramsey-like properties. In the last years, a new nonstandard technique based
on the use of iterated hyper-extensions has been developed to study
partition regularity of equations (see \cite{11dn-rado,11lu-pr}).

This paper aims at providing a self-contained introduction to a nonstandard
theory of ultrafilters; several examples are also included to illustrate the
use of such a theory in applications.

For gentle introductions to ultrafilters, see the papers \cite{11kt,11ga}; a
comprehensive reference is the monography \cite{11cn}. Recent surveys on
applications of ultrafilters across mathematics can be found in the book 
\cite{11nato}. As for nonstandard analysis, a short but rigorous
presentation can be found in \cite[\S 4.4]{11ck}; organic expositions
covering virtually all aspects of nonstandard methods are provided in the
books \cite{11nato,11lw,11lnl}. We remark that here we adopt the so-called 
\emph{external} approach, based on the existence of a \emph{star-map} $*$
that associates an \emph{hyper-extension} (or nonstandard extension) $%
{}^*\!A $ to each object $A$ under study, and satisfies the \emph{transfer
principle}. This is to be confronted to the \emph{internal} viewpoint as
formalized by E. Nelson's \emph{Internal Set Theory} \textsf{IST} or by K.
Hrb\'a\u{c}ek's \emph{Nonstandard Set Theories}. (See \cite{11kr} for a
thorough treatise of nonstandard set theories.)

Let us recall here the saturation property. A family $\mathcal{F}$ has the 
\emph{finite intersection property} (FIP for short) if $A_1\cap\ldots\cap
A_n\ne\emptyset$ for any choice of finitely many elements $A_i\in\mathcal{F}$%
.

\begin{definition}
\label{def-saturation} Let $\kappa$ be an infinite cardinal. A model of
nonstandard analysis is $\kappa$-\emph{saturated} if it satisfies the
property:

\begin{itemize}
\item Let $\mathcal{F}$ be a family of internal sets with cardinality $|%
\mathcal{F}|<\kappa$. If $\mathcal{F}$ has the FIP then $\bigcap_{A\in%
\mathcal{F}} A\ne\emptyset$.
\end{itemize}
\end{definition}

When $\kappa$-\emph{saturation} holds, then every infinite internal set $A$
has a cardinality $|A|\ge\kappa$. Indeed, the family of internal sets $%
\{A\setminus\{a\}\mid a\in A\}$ has the FIP, and has the same cardinality as 
$A$. If by contradiction $|A|<\kappa$, then by $\kappa$-\emph{saturation} we
would obtain $\bigcap_{a\in A}A\setminus\{a\}\ne\emptyset$, which is absurd.

With the exceptions of Sections \ref{sec-Hausdorff} and \ref{sec-regulargood}%
, throughout this paper we will work in a fixed $\mathfrak{c}^+$-saturated
model of nonstandard analysis, where $\mathfrak{c}$ is the cardinality of
the \emph{continuum}. (We recall that $\kappa^+$ denotes the successor
cardinal of $\kappa$. So, $\kappa^+$-saturation applies to families $|%
\mathcal{F}|\le\kappa$.) In consequence, our hypernatural numbers will have
cardinality $|{}^*\mathbb{N}|\ge\mathfrak{c}^+$.

\section{The $u$-equivalence}

\label{sec-uequivalence}

There is a canonical way of associating an ultrafilter on $\mathbb{N}$
to each hypernatural number.

\begin{definition}
The \emph{ultrafilter generated}%
\index{ultrafilter!generated} by a hypernatural number $\alpha%
\in{}^*\mathbb{N}$ is the family 
\begin{equation*}
\mathfrak{U}_\alpha=\{X\subseteq\mathbb{N}\mid \alpha\in{}^*\!X\}.
\end{equation*}
\end{definition}

It is easily verified that $\mathfrak{U}_\alpha$ actually satisfies the
properties of ultrafilter. Notice that $\mathfrak{U}_\alpha$ is principal if
and only if $\alpha\in\mathbb{N}$ is finite.

\begin{definition}
We say that $\alpha,\beta\in{}^*\mathbb{N}$ are $u$-\emph{equivalent}%
\index{u-equivalence}, and write $\alpha{\,{\sim}_{{}_{\!\!\!\!\! u}}\;}%
\beta $, if they generate the same ultrafilter, \emph{i.e.} $\mathfrak{U}%
_\alpha=\mathfrak{U}_\beta$. The equivalence classes $u(\alpha)=\{\beta\mid%
\beta{\,{\sim}_{{}_{\!\!\!\!\! u}}\;}\alpha\}$ are called $u$-\emph{monads}%
\index{u-monad}.
\end{definition}

Notice that $\alpha$ and $\beta$ are $u$-equivalent if and only if they
cannot be separated by any hyper-extension, \emph{i.e.} if $\alpha%
\in{}^*\!X\Leftrightarrow\beta\in{}^*\!X$ for every $X\subseteq\mathbb{N}$.
In consequence, the equivalence classes $u(\alpha)$ are characterized as
follows: 
\begin{equation*}
u(\alpha)\ =\ \bigcap\,\{{}^*\!X\mid X\in\mathfrak{U}_\alpha\}.
\end{equation*}

(The notion of \emph{filter monad} $\mu(\mathcal{F})=\bigcap\{{}^*F\mid F\in%
\mathcal{F}\}$ of a filter $\mathcal{F}$ was first introduced by W.A.J.
Luxemburg in \cite{11lux69}.)

For every ultrafilter $\mathcal{U}$ on $\mathbb{N}$, 
the family $\{{}^*\!X\mid X\in\mathcal{U%
}\}$ is a family of cardinality $\mathfrak{c}$ with the FIP and so, by $%
\mathfrak{c}^+$-\emph{saturation}, there exist hypernatural numbers $\alpha%
\in{}^*\mathbb{N}$ such that $\mathfrak{U}_\alpha=\mathcal{U}$. (Actually,
the $\mathfrak{c}^+$-\emph{enlargement} property suffices: see Definition %
\ref{def-enlargement}.) In consequence, 
\begin{equation*}
\beta\mathbb{N}\ =\ \{\mathfrak{U}_\alpha\mid\alpha\in\mathbb{N}\}.
\end{equation*}
Thus one can identify $\beta\mathbb{N}$ with the quotient set ${}^*\mathbb{N}%
/\!{\,{\sim}_{{}_{\!\!\!\!\! u}}\;}$ of the $u$-monads.

\begin{example}
Let $f:\mathbb{N}\to\mathbb{R}$ be bounded. If $\alpha{\,{\sim}%
_{{}_{\!\!\!\!\! u}}\;}\beta$ are $u$-equivalent then ${}^*\!f(\alpha)%
\approx{}^*\!f(\beta)$ are at infinitesimal distance. \newline
To see this, for every real number $r\in\mathbb{R}$ consider the set 
\begin{equation*}
\Gamma(r)\ =\ \{n\in\mathbb{N}\mid f(n)< r\}.
\end{equation*}
Then, by the hypothesis, one has $\alpha\in{}^*\Gamma(r)\Leftrightarrow\beta%
\in{}^*\Gamma(r)$, \emph{i.e.} ${}^*\!f(\alpha)< r\Leftrightarrow{}%
^*\!f(\beta)< r$. As this holds for all $r\in\mathbb{R}$, it follows that
the bounded hyperreal numbers ${}^*\!f(\alpha)\approx{}^*\!f(\beta)$ are
infinitely close. \newline
(This example was suggested to the author by E. Gordon.)
\end{example}

\begin{proposition}
\label{ueq1} ${}^*\!f\left(u(\alpha)\right)=u\left({}^*\!f(\alpha)\right)$.
Indeed:

\begin{enumerate}
\item If $\alpha{\,{\sim}_{{}_{\!\!\!\!\! u}}\;}\beta$ then ${}^*\!f(\alpha){%
\,{\sim}_{{}_{\!\!\!\!\! u}}\;}{}^*\!f(\beta)$.

\item If ${}^*\!f(\alpha){\,{\sim}_{{}_{\!\!\!\!\! u}}\;}\gamma$ then $%
\gamma={}^*\!f(\beta)$ for some $\beta{\,{\sim}_{{}_{\!\!\!\!\! u}}\;}\alpha$%
.
\end{enumerate}
\end{proposition}

\begin{proof}
$(1)$. For every $A\subseteq\mathbb{N}$, one has the following chain of
equivalences: 
\begin{equation*}
{}^*\!f(\alpha)\in{}^*\!A\ \Leftrightarrow\ \alpha\in{}^*\{n\mid f(n)\in
A\}\ \Leftrightarrow
\end{equation*}
\begin{equation*}
\Leftrightarrow\ \beta\in{}^*\{n\mid f(n)\in A\}\ \Leftrightarrow\
{}^*\!f(\beta)\in{}^*\!A.
\end{equation*}

$(2)$. For every $A\subseteq\mathbb{N}$, $\alpha\in{}^*\!A\Rightarrow{}%
^*\!f(\alpha)\in{}^*(f(A))\Leftrightarrow \gamma\in{}^*(f(A))$, \emph{i.e.} $%
\gamma={}^*\!f(\beta)$ for some $\beta\in{}^*\!A$. But then the family of
internal sets 
\begin{equation*}
\{{}^*\!f^{-1}(\gamma)\cap\,{}^*\!A\mid \alpha\in{}^*\!A\}
\end{equation*}
has the finite intersection property. By $\mathfrak{c}^+$-saturation, there
exists an element $\beta$ in the intersection of that family. Clearly, $%
{}^*\!f(\beta)=\gamma$ and $\beta{\,{\sim}_{{}_{\!\!\!\!\! u}}\;}\alpha$.
\end{proof}

Before starting to develop our nonstandard theory, let us consider a
well-known combinatorial property which constitutes a fundamental
preliminary step in the theory of ultrafilters. The proof given below
consists of two steps: we first show a finite version of the desired
property, and then use a non-principal ultrafilter to obtain the global
version. Although the result is well-known, this particular argument seems
to be new in the literature.

\begin{lemma}
\label{3-coloring} Let $f:\mathbb{N}\to\mathbb{N}$ be such that $f(n)\ne n$
for all $n$. Then there exists a 3-coloring $\chi:\mathbb{N}\to\{1,2,3\}$
such that $\chi(n)\ne\chi(f(n))$ for all $n$.
\end{lemma}

\begin{proof}
We begin by showing the following ``finite approximation" to the desired
result.

\begin{itemize}
\item \emph{For every finite $F\subset\mathbb{N}$ there exists $%
\chi_F:F\to\{1,2,3\}$ such that $\chi_F(x)\ne\chi_F(f(x))$ whenever both $x$
and $f(x)$ belong to $F$. }
\end{itemize}

We proceed by induction on the cardinality of $F$. The basis is trivial,
because if $|F|=1$ then it is never the case that both $x,f(x)\in F$. For
the inductive step, notice that by the \emph{pigeonhole principle} there
must be at least one element $\overline{x}\in F$ which is the image under $f$
of at most one element in $F$, \emph{i.e.} $|\{y\in F\mid f(y)=\overline{x}%
\}|\le 1$. Now let $F^{\prime }=F\setminus\{\overline{x}\}$ and let $%
\chi^{\prime }:F^{\prime }\to\{1,2,3\}$ be a 3-coloring as given by the
inductive hypothesis. We want to extend $\chi^{\prime }$ to a 3-coloring $%
\chi$ of $F$. To this end, define $\chi(\overline{x})$ in such a way that $%
\chi(\overline{x})\ne\chi^{\prime }(f(\overline{x}))$ if $f(\overline{x})\in
F$, and $\chi(\overline{x})\ne\chi^{\prime }(y)$ if $f(y)=\overline{x}$.
This is always possible because there is at most one such element $y$, and
because we have 3 colors at disposal.

We now have to glue together the finite 3-colorings so as to obtain a
3-coloring of the whole set $\mathbb{N}$. (Of course, this cannot be done
directly, because two 3-colorings do not necessarily agree on the
intersection of their domains.) One possible way is the following. For every 
$n\in\mathbb{N}$, fix a 3-coloring $\chi_n:\{1,\ldots,n\}\to\{1,2,3\}$ such
that $\chi_n(x)\ne\chi_n(f(x))$ whenever both $x,f(x)\in\{1,\ldots,n\}$.
Then pick any non-principal ultrafilter $\mathcal{U}$ on $\mathbb{N}$ and
define the map $\chi:\mathbb{N}\to\{1,2,3\}$ by putting 
\begin{equation*}
\chi(k)=i\ \Longleftrightarrow\ \Gamma_i(k)=\{n\ge k\mid \chi_n(k)=i\}\in%
\mathcal{U}.
\end{equation*}
The definition is well-posed because for every $k$ the disjoint union 
\begin{equation*}
\Gamma_1(k)\cup\Gamma_2(k)\cup\Gamma_3(k)=\{n\in\mathbb{N}\mid n\ge k\}\in%
\mathcal{U},
\end{equation*}
and so exactly one set $\Gamma_i(k)$ belongs to $\mathcal{U}$. The function $%
\chi$ is the desired 3-coloring. In fact, if by contradiction $%
\chi(k)=\chi(f(k))=i$ for some $k$, then we could pick $n\in\Gamma_i(k)\cap%
\Gamma_i(f(k))\in\mathcal{U}$ and have $\chi_n(k)=\chi_n(f(k))$, against the
hypothesis on $\chi_n$. (The same argument could be used to extend this
lemma to functions $f:I\to I$ over arbitrary infinite sets $I$.)
\end{proof}

\begin{remark}
The second part of the above proof could also be easily carried out by using
nonstandard methods. Indeed, by \emph{saturation} one can pick a hyperfinite
set $H\subset{}^*\mathbb{N}$ containing all (finite) natural numbers. By 
\emph{transfer} from the ``finite approximation" result proved above, there
exists an internal 3-coloring $\Phi:H\to\{1,2,3\}$ such that $%
\Phi(\xi)\ne\Phi({}^*\!f(\xi))$ whenever both $\xi,{}^*\!f(\xi)\in H$. Then
the restriction $\chi=\Phi\!\!\upharpoonleft_\mathbb{N}:\mathbb{N}%
\to\{1,2,3\}$ gives the desired 3-coloring.
\end{remark}

As a corollary, we obtain the

\begin{theorem}
\label{u-identity} Let $f:\mathbb{N}\to\mathbb{N}$ and $\alpha\in{}^*\mathbb{%
N}$. If ${}^*\!f(\alpha){\,{\sim}_{{}_{\!\!\!\!\! u}}\;}\alpha$ then $%
{}^*\!f(\alpha)=\alpha$.
\end{theorem}

\begin{proof}
If ${}^*\!f(\alpha)\ne\alpha$, then $\alpha\in{}^*\!A$ where $A=\{n\mid
f(n)\ne n\}$. Pick any function $g:\mathbb{N}\to\mathbb{N}$ that agrees with 
$f$ on $A$ and such that $g(n)\ne n$ for all $n\in\mathbb{N}$. Since $\alpha%
\in{}^*\!A\subseteq{}^*\{n\mid g(n)=f(n)\}$, we have that $%
{}^*\!g(\alpha)={}^*\!f(\alpha)$. Apply the previous theorem to $g$ and pick
a 3-coloring $\chi:\mathbb{N}\to\{1,2,3\}$ such that $\chi(n)\ne\chi(g(n))$
for all $n$. Then ${}^*\chi({}^*\!f(\alpha))={}^*\chi({}^*\!g(\alpha))\ne{}%
^*\chi(\alpha)$. Now let $X=\{n\in\mathbb{N}\mid\chi(n)=i\}$ where $%
i={}^*\chi(\alpha)$. Clearly, $\alpha\in{}^*\!X$ but ${}^*\!f(\alpha)\notin{}%
^*\!X$, and hence ${}^*\!f(\alpha){\,{\not\sim}_{{}_{\!\!\!\! u}}\,}\alpha$.
\end{proof}

Two important properties of $u$-equivalence are the following.

\begin{proposition}
Let $\alpha\in{}^*\!A$, and let $f$ be 1-1 when restricted to $A$. Then

\begin{enumerate}
\item There exists a bijection $\varphi$ such that ${}^*\!f(\alpha)={}^*\!%
\varphi(\alpha)$\,;

\item For every $g:\mathbb{N}\to\mathbb{N}$, ${}^*\!f(\alpha){\,{\sim}%
_{{}_{\!\!\!\!\! u}}\;}{}^*\!g(\alpha)\Rightarrow
{}^*\!f(\alpha)={}^*\!g(\alpha)$.
\end{enumerate}
\end{proposition}

\begin{proof}
$(1)$. We can assume that $\alpha\in{}^*\mathbb{N}\setminus\mathbb{N}$
infinite, as otherwise the thesis is trivial. Then $\alpha\in{}^*\!A$
implies that $A$ is infinite, and so we can partition $A=B\cup C$ into two
disjoint infinite sets $B$ and $C$ where, say, $\alpha\in{}^*B$. Since $f$
is 1-1, we can pick a bijection $\varphi$ that agrees with $f$ on $B$, so
that ${}^*\!\varphi(\alpha)={}^*\!f(\alpha)$ as desired.

$(2)$. By the previous point, ${}^*\!f(\alpha)={}^*\!\varphi(\alpha)$ for
some bijection $\varphi$. Then 
\begin{equation*}
{}^*\!g(\alpha){\,{\sim}_{{}_{\!\!\!\!\! u}}\;}{}^*\!\varphi(\alpha)\
\Rightarrow\ {}^*\!\varphi^{-1}({}^*\!g(\alpha)){\,{\sim}_{{}_{\!\!\!\!\!
u}}\;} {}^*\!\varphi^{-1}({}^*\!\varphi(\alpha))=\alpha\ \Rightarrow\
{}^*\!\varphi^{-1}({}^*\!g(\alpha))=\alpha,
\end{equation*}
and hence ${}^*\!g(\alpha)={}^*\!\varphi(\alpha)={}^*\!f(\alpha)$.
\end{proof}

We remark that property (2) of the above proposition
does not hold if we drop the hypothesis that $f$ is 1-1. (In Section \ref%
{sec-Hausdorff} we shall address the question of the existence of infinite
points $\alpha\in{}^*\mathbb{N}$ with the property that ${}^*\!f(\alpha){\,{%
\sim}_{{}_{\!\!\!\!\! u}}\;}{}^*\!g(\alpha)\Rightarrow{}^*\!f(\alpha)={}^*%
\!g(\alpha)$ for all $f,g:\mathbb{N}\to\mathbb{N}$.)

\begin{proposition}
\label{ueq2} If ${}^*\!f(\alpha){\,{\sim}_{{}_{\!\!\!\!\! u}}\;}\beta$ and $%
{}^*\!g(\beta){\,{\sim}_{{}_{\!\!\!\!\! u}}\;}\alpha$ for suitable $f$ and $%
g $, then ${}^*\!\varphi(\alpha){\,{\sim}_{{}_{\!\!\!\!\! u}}\;}\beta$ for
some bijection $\varphi$.
\end{proposition}

\begin{proof}
By the hypotheses, ${}^*\!g({}^*\!f(\alpha)){\,{\sim}_{{}_{\!\!\!\!\! u}}\;}%
{}^*\!g(\beta){\,{\sim}_{{}_{\!\!\!\!\! u}}\;}\alpha$ and so $%
{}^*\!g({}^*\!f(\alpha))=\alpha$. If $A=\{n\mid g(f(n))=n\}$, then $\alpha%
\in{}^*\!A$ and $f$ is 1-1 on $A$. By the previous proposition, there exists
a bijection $\varphi$ such that ${}^*\!f(\alpha)={}^*\!\varphi(\alpha)$, and
hence ${}^*\!\varphi(\alpha){\,{\sim}_{{}_{\!\!\!\!\! u}}\;}\beta$.
\end{proof}

We recall that the \emph{image} of an ultrafilter $\mathcal{U}$ under a
function $f:\mathbb{N}\to\mathbb{N}$ is the ultrafilter 
\begin{equation*}
f(\mathcal{U})\ =\ \{A\subseteq\mathbb{N}\mid f^{-1}(A)\in\mathcal{U}\}.
\end{equation*}
Notice that if $f\equiv_\mathcal{U} g$, \emph{i.e.} if $\{n\mid
f(n)=g(n)\}\in\mathcal{U}$, then $f(\mathcal{U})=g(\mathcal{U})$.

\begin{proposition}
For every $f:\mathbb{N}\to\mathbb{N}$ and $\alpha\in{}^*\mathbb{N}$, the
image ultrafilter $f(\mathfrak{U}_\alpha)=\mathfrak{U}_{{}^*\!f(\alpha)}$.
\end{proposition}

\begin{proof}
For every $A\subseteq\mathbb{N}$, one has the chain of equivalences: 
\begin{equation*}
A\in\mathfrak{U}_{{}^*\!f(\alpha)}\ \Leftrightarrow\ {}^*\!f(\alpha)\in{}%
^*\!A\ \Leftrightarrow\ \alpha\in{}^*(f^{-1}(A))\ \Leftrightarrow
\end{equation*}
\begin{equation*}
\Leftrightarrow\ f^{-1}(A)\in\mathfrak{U}_\alpha\ \Leftrightarrow\ A\in f(%
\mathfrak{U}_\alpha).
\end{equation*}
\end{proof}

Let us now show how the above results about $u$-equivalence are just
reformulation in a nonstandard context of fundamental properties of
ultrafilter theory.

\begin{theorem}
Let $f:\mathbb{N}\to\mathbb{N}$ and let $\mathcal{U}$ be an ultrafilter on $%
\mathbb{N}$. If $f(\mathcal{U})=\mathcal{U}$ then $\{n\mid f(n)=n\}\in%
\mathcal{U}$.
\end{theorem}

\begin{proof}
Let $\alpha\in{}^*\mathbb{N}$ be such that $\mathcal{U}=\mathfrak{U}_\alpha$%
. By the hypothesis, $\mathfrak{U}_\alpha=f(\mathfrak{U}_\alpha)=\mathfrak{U}%
_{\,{}^*\!f(\alpha)}$, \emph{i.e.} $\alpha{\,{\sim}_{{}_{\!\!\!\!\! u}}\;}%
{}^*\!f(\alpha)$ and so, by the previous theorem, ${}^*\!f(\alpha)=\alpha$.
But then $\{n\mid f(n)=n\}\in\mathcal{U}$ because $\alpha\in{}^*\{n\mid
f(n)=n\}$.
\end{proof}

Recall the \emph{Rudin-Keisler pre-ordering}%
\index{Rudin-Keisler pre-ordering} $\le_{RK}$ on ultrafilters: 
\begin{equation*}
\mathcal{V}\le_{RK}\mathcal{U}\ \Longleftrightarrow\ f(\mathcal{U})=\mathcal{%
V}\ \ 
\text{for some function }f.
\end{equation*}

In this case, we say that $\mathcal{V}$ is \emph{Rudin-Keisler below} $%
\mathcal{U}$ (or $\mathcal{U}$ is \emph{Rudin-Keisler above} $\mathcal{V}$).
It is readily verified that $g(f(\mathcal{U}))=(g\circ f)(\mathcal{U})$, so $%
\le_{RK}$ satisfies the transitivity property, and $\le_{RK}$ is actually a
pre-ordering. Notice that $\mathfrak{U}_\alpha\le_{RK}\mathfrak{U}_\beta$
means that ${}^*\!f(\beta){\,{\sim}_{{}_{\!\!\!\!\! u}}\;}\alpha$ for some
function $f$.

\begin{proposition}
$\mathcal{U}\le_{RK}\mathcal{V}$ and $\mathcal{V}\le_{RK}\mathcal{U}$ if and
only if $\mathcal{U}\cong \mathcal{V}$ are isomorphic, \emph{i.e.} there
exists a bijection $\varphi:\mathbb{N}\to\mathbb{N}$ such that $\varphi(%
\mathcal{U})=\mathcal{V}$.
\end{proposition}

\begin{proof}
Let $\mathcal{U}=\mathfrak{U}_\alpha$ and $\mathcal{V}=\mathfrak{U}_\beta$.
If $\mathcal{U}\le_{RK}\mathcal{V}$ and $\mathcal{V}\le_{RK}\mathcal{U}$,
then there exist functions $f,g:\mathbb{N}\to\mathbb{N}$ such that $%
{}^*\!f(\alpha){\,{\sim}_{{}_{\!\!\!\!\! u}}\;}\beta$ and ${}^*\!g(\beta){\,{%
\sim}_{{}_{\!\!\!\!\! u}}\;}\alpha$. But then, by Proposition \ref{ueq2},
there exists a bijection $\varphi:\mathbb{N}\to\mathbb{N}$ such that $%
{}^*\!\varphi(\alpha){\,{\sim}_{{}_{\!\!\!\!\! u}}\;}\beta$, and hence $%
\varphi(\mathcal{U})=\mathfrak{U}_{{}^*\!\varphi(\alpha)}=\mathfrak{U}_\beta=%
\mathcal{V}$, as desired. The other implication is trivial.
\end{proof}

We close this section by showing that all infinite numbers $\alpha$ have
``large" and ``spaced" $u$-monads, in the sense that $u(\alpha)$ is both a
left and a right unbounded subset of the infinite numbers ${}^*\mathbb{N}%
\setminus\mathbb{N}$, and that different elements of $u(\alpha)$ are placed
at infinite distance. (The property of $\mathfrak{c}^+$-\emph{saturation} is
essential here.)

\begin{theorem}
\label{monads}\cite{11pu1,11pu2} Let $\alpha\in{}^*\mathbb{N}\setminus%
\mathbb{N}$ be infinite. Then:

\begin{enumerate}
\item For every $\xi\in{}^*\mathbb{N}$, there exists an internal 1-1 map $%
\varphi:{}^*\mathbb{N}\to u(\alpha)\cap(\xi,+\infty)$. In consequence, the
set $u(\alpha)\cap(\xi,+\infty)$ contains $|{}^*\mathbb{N}|$-many elements
and it is unbounded in ${}^*\mathbb{N}$.

\item For every infinite $\xi\in{}^*\mathbb{N}\setminus\mathbb{N}$, the set $%
u(\alpha)\cap[0,\xi)$ contains at least $\mathfrak{c}^+$-many elements. In
consequence, $u(\alpha)$ is unbounded leftward in ${}^*\mathbb{N}\setminus%
\mathbb{N}$.

\item If $\alpha{\,{\sim}_{{}_{\!\!\!\!\! u}}\;}\beta$ and $\alpha\ne\beta$,
then the distance $|\alpha-\beta|\in{}^*\mathbb{N}\setminus\mathbb{N}$ is
infinite.
\end{enumerate}
\end{theorem}

\begin{proof}
(1).\ Since $\alpha$ is infinite, every $X\in\mathfrak{U}_\alpha$ is an
infinite set and so for each $k\in\mathbb{N}$ there exists a 1-1 function $f:%
\mathbb{N}\to X\cap(k,+\infty)$. By \emph{transfer}, for every $\xi\in{}^*%
\mathbb{N}$ the following internal set is non-empty: 
\begin{equation*}
\Gamma(X)=\{\,\varphi:{}^*\mathbb{N}\to{}^*\!X\cap(\xi,+\infty)\mid \varphi 
\text{\ internal 1-1}\,\}.
\end{equation*}
Notice that $\Gamma(X_1)\cap\cdots\cap\Gamma(X_n)=$ $\Gamma(X_1\cap\cdots%
\cap X_n)$, and hence the family $\{\,\Gamma(X)\mid X\in\mathfrak{U}%
_\alpha\} $ has the finite intersection property. By $\mathfrak{c}^+$-\emph{%
saturation}, we can pick $\varphi\in\bigcap_{X\in\mathfrak{U}%
_\alpha}\Gamma(X)$. Clearly, $\text{range}(\varphi)$ is an internal subset
of $u(\alpha)\cap(\xi,+\infty)$ with the same cardinality as ${}^*\mathbb{N}$%
. Since $\text{range}(\varphi)$ is internal and hyperinfinite, it is
necessarily unbounded in ${}^*\mathbb{N}$.

(2).\ For any given $\xi\in{}^*\mathbb{N}\setminus\mathbb{N}$, the family $%
\{{}^*\!X\cap[0,\xi)\mid X\in\mathfrak{U}_\alpha\}$ is closed under finite
intersections, and all its elements are non-empty. So, by $\mathfrak{c}^+$-%
\emph{saturation}, there exists 
\begin{equation*}
\zeta\in\bigcap_{X\in\mathfrak{U}_\alpha}{}^*\!X\cap[0,\xi).
\end{equation*}
Clearly $\zeta\in u(\alpha)\cap[0,\xi)$, and this shows that $u(\alpha)$ is
unbounded leftward in ${}^*\mathbb{N}\setminus\mathbb{N}$. Now fix $\xi$
infinite. By what we have just proved, the family of open intervals 
\begin{equation*}
\mathcal{G}=\{(k,\zeta)\mid k\in\mathbb{N} \text{ and } \zeta\in
u(\alpha)\cap[0,\xi)\}
\end{equation*}
has empty intersection. Since $\mathcal{G}$ satisfies the finite
intersection property, and $\mathfrak{c}^+$-\emph{saturation} holds, it must
be $|\mathcal{G}|\ge\mathfrak{c}^+$, and hence also $|u(\alpha)\cap[0,\xi)|%
\ge\mathfrak{c}^+$.

(3).\ For every $n\ge 2$, let $k_n$ be the remainder of the Euclidean
division of $\alpha$ by $n$, and consider the set $X_n=\{x\cdot n+k_n\mid
x\in\mathbb{N}\}$. Then $\alpha\in{}^*\!X_n$ and $\alpha{\,{\sim}%
_{{}_{\!\!\!\!\! u}}\;}\beta$ implies that also $\beta\in{}^*\!X_n$, so $%
\alpha-\beta$ is a multiple of $n$. Since $\beta\ne\alpha$, it must be $%
|\alpha-\beta|\ge n$. As this is true for all $n\ge 2$, we conclude that $%
\alpha$ and $\beta$ have infinite distance.
\end{proof}

\section{Hausdorff S-topologies and Hausdorff ultrafilters}

\label{sec-Hausdorff}

It is natural to ask about properties of the \emph{ultrafilter map}%
\index{ultrafilter map}: 
\begin{equation*}
\mathfrak{U}:{}^*\mathbb{N}\to\beta\mathbb{N}\quad%
\text{where}\quad \mathfrak{U}:\alpha\mapsto\mathfrak{U}_\alpha.
\end{equation*}
We already noticed that if one assumes $\mathfrak{c}^+$-\emph{saturation}
then $\mathfrak{U}$ is onto $\beta\mathbb{N}$, \emph{i.e.} every ultrafilter
on $\mathbb{N}$ is of the form $\mathfrak{U}_\alpha$ for a suitable $\alpha%
\in{}^*\mathbb{N}$. However, in this section no saturation property will be
assumed.

As a first (negative) result, let us show that the ultrafilter map is never
a bijection.

\begin{proposition}
In any model of nonstandard analysis, if the ultrafilter map $\mathfrak{U}%
:{}^*\mathbb{N}\twoheadrightarrow\beta\mathbb{N}$ is onto then, for every
non-principal $\mathcal{U}\in\beta\mathbb{N}$, the set $\{\alpha\in{}^*%
\mathbb{N}\mid\mathfrak{U}_\alpha=\mathcal{U}\}$ contains at least $%
\mathfrak{c}$-many elements.
\end{proposition}

\begin{proof}
Given a non-principal ultrafilter $\mathcal{U}$ on $\mathbb{N}$, for $X\in%
\mathcal{U}$ and $k\in\mathbb{N}$ let 
\begin{equation*}
\Lambda(X,k)\ =\ \left\{F\in\text{Fin}(\mathbb{N})\mid F\subset X\ \&\
|F|\ge k\right\},
\end{equation*}
where we denoted by $\text{Fin}(\mathbb{N})=\{F\subset\mathbb{N}\mid F\ 
\text{is finite}\}$. Notice that the family of sets $\mathcal{F}%
=\{\Lambda(X,k)\mid X\in\mathcal{U}\,,\,k\in\mathbb{N}\}$ has the finite
intersection property. Indeed, $\Lambda(X_1,k_1)\cap\ldots\cap%
\Lambda(X_m,k_m)=\Lambda(X,k)$ where $X=X_1\cap\ldots\cap X_m\in\mathcal{U}$
and $k=\max\{k_1,\ldots,k_m\}$; and every set $\Lambda(X,k)\ne\emptyset$
since all $X\in\mathcal{U}$ are infinite. Now fix a bijection $\Phi:\text{Fin%
}(\mathbb{N})\to\mathbb{N}$, and let 
\begin{equation*}
\Gamma(X,k)\ =\ \{\Psi(F)\mid F\in\Lambda(X,k)\}.
\end{equation*}
Then also the family $\{\Gamma(X,k)\mid X\in\mathcal{U}\,,\,k\in\mathbb{N}%
\}\subseteq\mathcal{P}(\mathbb{N})$ has the FIP, and so we can extend it to
an ultrafilter $\mathcal{V}$ on $\mathbb{N}$. By the hypothesis on the
ultrafilter map there exists $\beta\in{}^*\mathbb{N}$ such that $\mathfrak{U}%
_\beta=\mathcal{V}$; in particular, $\beta\in\bigcap_{X,k}{}^*\Gamma(X,k)$,
and so $\beta={}^*\!\left(\Psi(G)\right)$ for a suitable $%
G\in\bigcap_{X,k}{}^*\Lambda(X,k)$. Then $G\subseteq{}^*\!X$ for all $X\in%
\mathcal{U}$, and hence $\mathfrak{U}_\gamma=\mathcal{U}$ for all $\gamma\in
G$. Moreover, $|G|\ge k$ for all $k\in\mathbb{N}$, and so $G$ is an infinite
internal set. Finally, we use the following general fact: ``Every infinite
internal set has at least the cardinality of the continuum". To prove this
last property, notice that if $A$ is infinite and internal then there exists
a (internal) 1-1 map $f:\{1,\ldots,\nu\}\to A$ for some infinite $\nu\in{}^*%
\mathbb{N}\setminus\mathbb{N}$. Now, consider the unit real interval $[0,1]$
and define $\Psi:[0,1]\to\{1,\ldots,\nu\}$ by putting $\Psi(r)=\min\{1\le
i\le\nu\mid r\le i/\nu\}$. The map $\Psi$ is 1-1 because $%
\Psi(r)=\Psi(r^{\prime })\Rightarrow |r-r^{\prime }|\le 1/\nu\approx
0\Rightarrow r=r^{\prime }$, and so we conclude that $\mathfrak{c}%
=|[0,1]|\le|\{1,\ldots,\nu\}|\le|A|$, as desired. (When $\mathfrak{c}^+$-%
\emph{saturation holds}, then $|\{\alpha\in{}^*\mathbb{N}\mid\mathfrak{U}%
_\alpha=\mathcal{U}\}|\ge\mathfrak{c}^+$ by Theorem \ref{monads}.)
\end{proof}

We now show that the ultrafilter map is tied up with a topology that is
naturally considered in a nonstandard setting. (The notion of $S$-topology
was introduced by A. Robinson himself, the ``inventor" of nonstandard
analysis.)

\begin{definition}
\label{S-topology} For every set $X$, the \emph{S-topology}%
\index{S-topology} on ${}^*\!X$ is the topology having the family $%
\{{}^*\!A\mid A\subseteq X\}$ as a basis of open sets.
\end{definition}

The capital letter ``S" stands for ``standard", and in fact hyper-extensions 
${}^*\!A$ are often called \emph{standard sets} in the literature of
nonstandard analysis. The adjective ``standard" originated from the
distinction between a \emph{standard} universe and a \emph{nonstandard}
universe, according to the most used approaches to nonstandard analysis.
However, such a distinction is not needed, and indeed one can adopt a
foundational framework where there is a single mathematical universe, and
take hyper-extensions of \emph{any} object under study (see, \emph{e.g.}, 
\cite{11bh}).

Every basic open set ${}^*\!A$ is also closed because ${}^*\!X%
\setminus{}^*\!A={}^*(X\setminus A)$, and so the $S$-topologies are totally
disconnected. A first relationship between $S$-topology and ultrafilter map
is the following.

\begin{proposition}
\label{ultramaponto} The $S$-topology on ${}^*\mathbb{N}$ is compact if and
only if the ultrafilter map $\mathfrak{U}:{}^*\mathbb{N}\to\beta\mathbb{N}$
is onto.
\end{proposition}

\begin{proof}
According to one of the equivalent definitions of compactness, the $S$%
-topology is compact if and only if every non-empty family $\mathcal{C}$ of
closed sets with the FIP has non-empty intersection $\bigcap_{C\in\mathcal{C}%
}C\ne\emptyset$. Without loss of generality, one can assume that $\mathcal{C}
$ is a family of hyper-extensions. Notice that $\mathcal{C}=\{{}^*\!A_i\mid
i\in I\}$ has the FIP if and only if $\mathcal{C}^{\prime }=\{A_i\mid i\in
I\}\subset\mathcal{P}(\mathbb{N})$ has the FIP, and so we can extend $%
\mathcal{C}^{\prime }$ to an ultrafilter $\mathcal{V}$ on $\mathbb{N}$. If
the ultrafilter map is onto $\beta\mathbb{N}$, then $\mathcal{V}=\mathfrak{U}%
_\alpha$ for a suitable $\alpha$, and therefore $\alpha\in\bigcap_{i\in
I}{}^*\!A_i\ne\emptyset$.

Conversely, if $\mathcal{U}$ is an ultrafilter on $\mathbb{N}$, then $%
\mathcal{C}=\{{}^*\!X\mid X\in\mathcal{U}\}$ is a family of closed sets with
the FIP. If $\alpha$ is any element in the intersection of $\mathcal{C}$,
then $\mathfrak{U}_\alpha=\mathcal{U}$.
\end{proof}

In consequence of the above proposition, the $S$-topology on ${}^*\mathbb{N}$
is compact when the $\mathfrak{c}^+$-\emph{saturation} property holds. More
generally, $\kappa$-\emph{saturation} implies that the $S$-topology is
compact on every hyper-extension ${}^*\!X$ where $2^{|X|}<\kappa$.
(Actually, the $\kappa$-\emph{enlarging} property suffices: see Definition %
\ref{def-enlargement}.)

A natural question that one may ask is whether the S-topologies are
Hausdorff or not. This depends on the considered model, and giving a
complete answer turns out to be a difficult issue involving deep
set-theoretic matters, which will be briefly discussed below. So, it is not
surprising that this simple question was not addressed explicitly in the
early literature in nonstandard analysis, despite the fact that the
S-topology was a common object of study.

As a first remark, notice that having a Hausdorff S-topology on ${}^*\!X$ is
preserved when passing to lower cardinalities.

\begin{proposition}
\label{Hausdorffbelow} If the S-topology is Hausdorff on ${}^*\!X$ and $%
|Y|\le|X|$, then the S-topology is Hausdorff on ${}^*Y$ as well.
\end{proposition}

\begin{proof}
Fix a 1-1 map $f:Y\to X$. Given $\xi\ne\eta$ in ${}^*Y$, consider $%
{}^*\!f(\xi)\ne{}^*\!f(\eta)$ in ${}^*\!X$. By the hypothesis, we can pick
disjoint sets $A,B\subseteq X$ with ${}^*\!f(\xi)\in{}^*\!A$ and $%
{}^*\!f(\eta)\in{}^*B$. Then $C=f^{-1}(A)$ and $D=f^{-1}(B)$ are disjoint
subsets of $Y$ such that $\xi\in{}^*C$ and $\eta\in{}^*D$.
\end{proof}

Recall a notion that was introduced in \cite{11hr}: A model of nonstandard
analysis is $\kappa$-\emph{constrained} if the following property holds: 
\begin{equation*}
\forall X\ \forall \xi\in{}^*\!X\ \exists A\subseteq X\ \text{such that}\
|A|\le\kappa\ \text{and}\ \xi\in{}^*\!A.
\end{equation*}

We remark that any ultrapower model of nonstandard analysis constructed by
means of an ultrafilter over a set of cardinality $\kappa$ is $\kappa$%
-constrained.

In the countable case, the notion of \emph{constrained} already appeared at
the beginnings of nonstandard analysis, under the name of $\sigma$-\emph{%
quasi standardness} (see W.A.J. Luxemburg's lecture notes \cite{11lux62}).
The existence of nonstandard universes which are not $\kappa$-constrained
for any $\kappa$ is problematic and appears to be closely related to the
existence large cardinals. The reader interested in this foundational issue
is referred to \cite{11hr}.

\begin{proposition}
Assume that our model of nonstandard analysis is $\kappa$-constrained. Then
the S-topology is Hausdorff on ${}^*\kappa$ if and only if the S-topology is
Hausdorff on every hyper-extension ${}^*\!X$.
\end{proposition}

\begin{proof}
Let $\xi,\eta\in{}^*\!X$ be given. By the property of $\kappa$-constrained,
we can pick sets $A,B$ with $\xi\in{}^*\!A$, $\eta\in{}^*B$ and $%
|A|,|B|\le\kappa$. Since $|A\cup B|\le\kappa$, by the previous Proposition %
\ref{Hausdorffbelow}, the S-topology on ${}^*(A\cup B)$ is Hausdorff. Pick
disjoint subsets $C,D\subseteq A\cup B$ with $\xi\in{}^*C$ and $\eta\in{}^*D$%
. Then $\xi\in{}^*(C\cap X)$ and $\eta\in{}^*(D\cap X)$ where $C\cap X$ and $%
D\cap X$ are disjoint subsets of $X$.
\end{proof}

So, in any ultrapower model of nonstandard analysis determined by an
ultrafilter on $\mathbb{N}$, if the S-topology is Hausdorff on ${}^*\mathbb{N%
}$ then it is Hausdorff on \emph{all} hyper-extensions ${}^*\!X$.

\begin{proposition}
\label{ultramap11} The $S$-topology on ${}^*\!X$ is Hausdorff if and only if
the ultrafilter map $\mathfrak{U}:{}^*\!X\to\beta X$ is 1-1.
\end{proposition}

\begin{proof}
By definition, the $S$-topology on ${}^*\!X$ is Hausdorff if and only for
every pair of elements $\xi\ne\eta$ in ${}^*\!X$ there exist basic open sets 
${}^*\!A,{}^*B\subseteq{}^*\!X$ such that $\xi\in{}^*\!A$, $\eta\in{}^*B$
and ${}^*\!A\cap{}^*B=\emptyset$. But this means that $A\in\mathfrak{U}_\xi$
and $B\in\mathfrak{U}_\eta$ for suitable disjoint sets $A\cap B=\emptyset$.
We reach the thesis by noticing that this last property holds if and only if
the ultrafilters $\mathfrak{U}_\xi$ and $\mathfrak{U}_\eta$ are different.
\end{proof}

Hausdorff $S$-topologies are tied up with special ultrafilters.

\begin{proposition}
Let $\mathcal{U}$ be a non-principal ultrafilter on the set $I$. Then the
following are equivalent:

\begin{enumerate}
\item In the ultrapower model of nonstandard analysis determined by $%
\mathcal{U}$, the $S$-topology on ${}^*I$ is Hausdorff.

\item In any model of nonstandard analysis, if $\mathcal{U}=\mathfrak{U}%
_\alpha$ is generated by a point $\alpha\in{}^*I$, then: 
\begin{equation*}
{}^*\!f(\alpha){\,{\sim}_{{}_{\!\!\!\!\! u}}\;}{}^*\!g(\alpha)\
\Longrightarrow\ {}^*\!f(\alpha)={}^*\!g(\alpha).
\end{equation*}

\item For every $f,g:I\to I$, 
\begin{equation*}
f(\mathcal{U})=g(\mathcal{U})\ \Longrightarrow\ f\equiv_\mathcal{U} g,\ 
\emph{i.e.}\ \{i\in I\mid f(i)= g(i)\}\in\mathcal{U}.
\end{equation*}
\end{enumerate}
\end{proposition}

\begin{proof}
$(1)\Leftrightarrow(3)$. Notice first that if $\xi=[f]_\mathcal{U}$ is the
element of ${}^*I$ given by the $\mathcal{U}$-equivalence class of the
function $f:I\to I$, then $f(\mathcal{U})=\mathfrak{U}_\xi$. Indeed, for
every $A\subseteq I$, one has $A\in f(\mathcal{U})\Leftrightarrow
f^{-1}(A)\in\mathcal{U} \Leftrightarrow\{i\in I\mid f(i)\in A\}\in\mathcal{U}%
\Leftrightarrow \xi=[f]_\mathcal{U}\in{}^*\!A\Leftrightarrow A\in\mathfrak{U}%
_\xi$. Now let $\xi=[f]_\mathcal{U}$ and $\eta=[g]_\mathcal{U}$ be arbitrary
elements of ${}^*I=I^I/\mathcal{U}$. By definition, $\xi=\eta\Leftrightarrow
f\equiv_\mathcal{U} g$; besides, by what just seen above, $f(\mathcal{U})=g(%
\mathcal{U})\Leftrightarrow \mathfrak{U}_\xi=\mathfrak{U}_\eta$. Now, by the
previous proposition the $S$-topology on $I$ is Hausdorff if and only if the
ultrafilter map on $I$ is 1-1, and hence the thesis follows.

$(2)\Leftrightarrow(3)$. Notice that ${}^*\!f(\alpha)={}^*\!g(\alpha)%
\Leftrightarrow\alpha\in{}^*\{i\in I\mid f(i)=g(i)\} \Leftrightarrow
f\equiv_{\mathfrak{U}_\alpha}g$. The thesis follows by recalling that $%
\mathfrak{U}_{{}^*\!f(\alpha)}=f(\mathfrak{U}_\alpha)$ and $\mathfrak{U}%
_{{}^*\!g(\alpha)}=g(\mathfrak{U}_\alpha)$.
\end{proof}

Because of the above equivalences, non-principal ultrafilters that satisfy
property $(3)$, were named \emph{Hausdorff} in \cite{11dfH}. To the author's
knowledge, the problem of existence of such ultrafilters was first
explicitly considered by A. Connes in his paper \cite{11co} of 1970, where
he needed special ultrafilters $\mathcal{U}$ with the property that the maps 
$[\varphi]_\mathcal{U}\mapsto\varphi(\mathcal{U})$ defined on ultrapowers $K^%
\mathbb{N}/\mathcal{U}$ ($K$ a field) be injective into $\beta\mathbb{N}$.
He noticed that such a property was satisfied by selective ultrafilters,
introduced three years before by G. Choquet \cite{11cho} under the name of 
\emph{ultrafiltres absolus}. In consequence, Hausdorff ultrafilters are
consistent. Indeed, selective ultrafilters exist under the \emph{continuum
hypothesis} (this was already proved by G. Choquet \cite{11cho} in 1968).
However, we remark that the existence of selective ultrafilters cannot be
proved in \textsf{ZFC}, as first shown by K. Kunen \cite{11ku}.
Independently, in their 1972 paper \cite{11ch}, G. Cherlin and J. Hirschfeld
proved that non-principal ultrafilters exist which are \emph{not} Hausdorff,
and asked whether Hausdorff ultrafilters exist at all in \textsf{ZFC}. It is
worth remarking that this problem is still open to this day (see \cite%
{11dfH,11bs}).

We close this section by mentioning another result, proved in \cite{11dfS},
that connects Hausdorff ultrafilters and nonstandard analysis

\begin{theorem}
Assume that $\mathfrak{N}\subset\beta\mathbb{N}$ is a set of ultrafilters on 
$\mathbb{N}$ such that

\begin{itemize}
\item $\mathbb{N}\varsubsetneq\mathfrak{N}$, \emph{i.e.} $\mathfrak{N}$
properly contains all principal ultrafilters\,;

\item Every non-principal $\mathcal{U}\in\mathfrak{N}$ is Hausdorff\,;

\item $\mathfrak{N}$ is RK-downward closed, \emph{i.e.} $\mathcal{U}\in 
\mathfrak{N}$ implies that $f(\mathcal{U})\in\mathfrak{N}$ for every $f:%
\mathbb{N}\to\mathbb{N}$\,;

\item $\mathfrak{N}$ is ``strongly" RK-filtered in the following sense: For
every $\mathcal{U},\mathcal{V}\in\mathfrak{N}$ there exist $\mathcal{W}\in%
\mathfrak{N}$ and $f,g:\mathbb{N}\to\mathbb{N}$ such that $f(\mathcal{W})=%
\mathcal{U}$ and $g(\mathcal{W})=\mathcal{V}$.
\end{itemize}

Then $\mathfrak{N}$ is a set of hypernatural numbers of nonstandard analysis
where:

\begin{itemize}
\item For $A\subseteq\mathbb{N}^k$, the hyper-extension ${}^*\!A\subseteq%
\mathfrak{N}^k$ is defined by letting for every $\mathcal{U}\in\mathfrak{N}$ and for
every $f_1,\ldots,f_k:\mathbb{N}\to\mathbb{N}$: 
\begin{equation*}
(f_1(\mathcal{U}),\ldots,f_k(\mathcal{U}))\in{}^*\!A\ \Longleftrightarrow\
\{n\in\mathbb{N}\mid (f_1(n),\ldots,f_k(n))\in A\}\in\mathcal{U}
\end{equation*}

\item For $F:\mathbb{N}^k\to\mathbb{N}$, the hyper-extension ${}^{*}\!F :%
\mathfrak{N}^k\to\mathfrak{N}$ is defined by letting for every $\mathcal{U}%
\in\mathfrak{N}$ and for every $f_1,\ldots,f_k:\mathbb{N}\to\mathbb{N}$: 
\begin{equation*}
{}^{*}\!F (f_1(\mathcal{U}),\ldots,f_k(\mathcal{U}))\ =\
(F\circ(f_1,\ldots,f_k))(\mathcal{U}).
\end{equation*}
($F\circ(f_1,\ldots,f_k):\mathbb{N}\to\mathbb{N}$ is the function $n\mapsto
F(f_1(n),\ldots,f_k(n))$.)
\end{itemize}
\end{theorem}

\section{Regular and good ultrafilters}

\label{sec-regulargood}

A fundamental notion used in the theory of ultrafilters is that of
regularity. We recall that an ultrafilter $\mathcal{U}$ on an infinite set $%
I $ is called \emph{regular}%
\index{ultrafilter!regular} if there exists a family $\{A_i\mid i\in
I\}\subseteq\mathcal{U}$ such that $\bigcap_{i\in I_0}A_i=\emptyset$ for
every infinite $I_0\subseteq I$. When $I$ is countable, it is easily seen
that $\mathcal{U}$ is regular if and only if it is non-principal, but in
general regularity is a stronger condition. A simple nonstandard
characterization holds.

Recall the following weakened version of saturation, where only families of
hyper-extensions are considered (compare with Definition \ref{def-saturation}%
.)

\begin{definition}
\label{def-enlargement} Let $\kappa$ be an infinite cardinal. A model of
nonstandard analysis is a $\kappa$-\emph{enlargement} if it satisfies the
property:

\begin{itemize}
\item Let $\mathcal{G}$ be a family of sets with cardinality $|\mathcal{G}%
|<\kappa$. If $\mathcal{G}$ has the FIP, then $\bigcap_{B\in\mathcal{G}%
}{}^*B\ne\emptyset$
\end{itemize}
\end{definition}

\begin{proposition}
\label{regularity} Let $\mathcal{U}$ be an ultrafilter over the infinite set 
$I$. Then the following are equivalent:

\begin{enumerate}
\item $\mathcal{U}$ is regular.

\item In the ultrapower model of nonstandard analysis determined by $%
\mathcal{U}$, the $|I|^+$-enlarging property holds.
\end{enumerate}
\end{proposition}

\begin{proof}
$(1)\Rightarrow(2)$. Pick a family $\{C_x\mid x\in I\}\subseteq\mathcal{U}$
such that $\bigcap_{x\in\Lambda}C_x=\emptyset$ whenever $\Lambda\subseteq I$
is infinite. Given a family $\{A_x\mid x\in I\}$ with the FIP, for every $%
y\in I$ pick an element $\varphi(y)\in\bigcap_{y\in C_x}A_x$ (this is
possible because $\{x\mid y\in C_x\}$ is finite). If $\xi=[\varphi]_\mathcal{%
U}$ is the element given by the $\mathcal{U}$-equivalence class of the
function $\varphi$, then $\xi%
\in{}^*\!A_x$ for every $x\in I$, since $\{y\in I\mid \varphi(y)\in
A_x\}\supseteq\{y\in I\mid y\in C_x\}=C_x\in\mathcal{U}$.

$(2)\Rightarrow(1)$. Let $\text{Fin}(I)$ be the set of finite parts of $I$,
and for every $x\in I$, let $\widehat{x}=\{a\in \text{Fin}(I)\mid x\in a\}$.
The family $\{\widehat{x}\mid x\in I\}$ satisfies the FIP and so, by the $%
\kappa^+$-enlarging property, there exists $A\in\bigcap_{x\in I}{}^*\widehat{%
x}$. Now pick a function $\varphi:I\to\text{Fin}(I)$ such that $A=[\varphi]_%
\mathcal{U}$ is the $\mathcal{U}$-equivalence class of $\varphi$. Then for
every $x\in I$, one has that $A\in{}^*\widehat{x}$ if and only if 
\begin{equation*}
C_x\ =\ \{y\in I\mid\varphi(y)\in\widehat{x}\}\ =\ \{y\in I\mid
x\in\varphi(y)\}\,\in\,\mathcal{U}.
\end{equation*}
The family $\{C_x\mid x\in I\}\subseteq\mathcal{U}$ witnesses the regularity
of $\mathcal{U}$ because for every infinite $\Lambda\subseteq I$, the
intersection $\bigcap_{x\in\Lambda}C_x=\{y\in I\mid
\Lambda\subseteq\varphi(y)\}=\emptyset$.
\end{proof}

One can also easily characterize points that generate regular ultrafilters.

\begin{theorem}
\label{regularity2} Given any model of nonstandard analysis, let $\alpha\in{}%
^*I$. Then the following are equivalent:

\begin{enumerate}
\item The ultrafilter $\mathfrak{U}_\alpha$ on $I$ generated by $\alpha$ is
regular\,;

\item There exists a function $\varphi:I\to\text{Fin}(I)$ such that ${}^*x%
\in{}^*\!\varphi(\alpha)$ for all $x\in I$.
\end{enumerate}
\end{theorem}

\begin{proof}
$(1)\Rightarrow(2)$. Given a family $\{C_x\mid x\in I\}$ that witnesses the
regularity of $\mathfrak{U}_\alpha$, let $\varphi:I\to\text{Fin}(I)$ be the
function defined by setting $\varphi(x)=\{y\in I\mid x\in C_y\}$. If $%
\vartheta$ is the function $y\mapsto C_y$, then we obtain the thesis by the
following equivalences, that hold for any $x\in I$: 
\begin{equation*}
C_x\in\mathfrak{U}_\alpha\ \Leftrightarrow\ \alpha\in{}^*(C_x)={}^*%
\vartheta({}^*x)\ \Leftrightarrow\ {}^*x\in\{\xi\in{}^*I\mid \alpha\in{}%
^*\vartheta(\xi)\}={}^*\varphi(\alpha).
\end{equation*}

$(2)\Rightarrow(1)$. Pick $\varphi$ as in the hypothesis, and let 
\begin{equation*}
C_x\ =\ \{y\in I\mid x\in\varphi(y)\}.
\end{equation*}
Notice that ${}^*x\in{}^*\!\varphi(\alpha)\Leftrightarrow \alpha\in{}%
^*(C_x)\Leftrightarrow C_x\in\mathfrak{U}_\alpha$. Moreover, for $X\subseteq
I$, whenever it is possible to pick an element $y\in\bigcap_{x\in X}C_x$,
one has that $X\subseteq\varphi(y)$, and so $X$ is finite.
\end{proof}

As a corollary of the above characterizations of regularity, we can now give
a nonstandard proof of a limiting result about the existence of Hausdorff
ultrafilters.

We recall that the \emph{ultrafilter number} $\mathfrak{u}$ denotes the
minimum cardinality of any (non-principal) ultrafilter base on $\mathbb{N}$.
(A family $\mathcal{B}$ is an \emph{ultrafilter base} on $\mathbb{N}$ if $%
\{A\subseteq\mathbb{N}\mid \exists B\in\mathcal{B},\ B\subseteq A\}$ is an
ultrafilter.) It is not hard to show that $\aleph_0<\mathfrak{u}\le\mathfrak{%
c}$ (see \cite[\S 9]{11blh}).

\begin{theorem}
\cite{11dfH}\ If $\mathcal{U}$ is a regular ultrafilter on a cardinal $%
\kappa\ge\mathfrak{u}$, then $\mathcal{U}$ is not Hausdorff.
\end{theorem}

\begin{proof}
By contradiction, pick $\mathcal{U}$ a regular Hausdorff ultrafilter on $%
\kappa$, and consider the corresponding ultrapower model of nonstandard
analysis. Then the $S$-topology on ${}^*\kappa$, and hence on ${}^*\mathbb{N}
$, is Hausdorff. Moreover, by the characterization of Proposition \ref%
{regularity}, the $\kappa^+$-enlarging property holds. Now pick $\mathcal{G}%
\subset\mathcal{P}(\mathbb{N})$ a non-principal ultrafilter base of
cardinality $\mathfrak{u}$, and for every $A\in\mathcal{G}$, let $%
\Gamma(A)=\{X\subseteq A\mid X\ \text{is infinite}\}$. Clearly the family $%
\{\Gamma(A)\mid A\in\mathcal{G}\}$ has the FIP, and so by the enlarging
property there exists $H\in\bigcap_{A\in\mathcal{G}}{}^*\Gamma(A)$. Then $H%
\subset{}^*\mathbb{N}$ is hyperinfinite and $\mathfrak{U}_\xi=\mathcal{V}$
for every $\xi\in H$, where $\mathcal{V}$ is the ultrafilter on $\mathbb{N}$
generated by $\mathcal{G}$. This shows that the ultrafilter map $\mathfrak{U}%
:{}^*\mathbb{N}\to\beta\mathbb{N}$ is not 1-1, and hence the $S$-topology on 
${}^*\mathbb{N}$ is not Hausdorff, a contradiction.
\end{proof}

We recall that an ultrafilter $\mathcal{U}$ is \emph{countably incomplete}
if it is not closed under countable intersections, \emph{i.e.} if there
exists a countable family $\{A_n\}$ of elements of $\mathcal{U}$ such that $%
\bigcap_n A_n\notin\mathcal{U}$ (equivalently, one may ask for that
intersection to be empty). We remark that an ultrapower modulo $\mathcal{U}$
determines a model of nonstandard analysis if and only if $\mathcal{U}$ is
countably incomplete, as otherwise one would have ${}^*\mathbb{N}=\mathbb{N}$%
. Indeed, if $\xi=[f]_\mathcal{U}$ is an infinite element in the ultrapower $%
{}^*\mathbb{N}=\mathbb{N}^I/\mathcal{U}$, then for every $n\in\mathbb{N}$,
the set $\Gamma_n=\{i\in\mathbb{N}\mid f(i)\ne n\}\in\mathcal{U}$ but the
countable intersection $\bigcap_n\Gamma_n=\emptyset$. An ultrafilter $%
\mathcal{U}$ on an infinite set $I$ is \emph{good}%
\index{ultrafilter!good} if for every antimonotonic function $f:%
\text{Fin}(I)\to\mathcal{U}$ there exists an antiadditive $g:\text{Fin}(I)\to%
\mathcal{U}$ such that $g(a)\subseteq f(a)$ for all $a\in\text{Fin}(I)$. $f$
is called \emph{antimonotonic} if $f(a)\supseteq f(b)$ whenever $a\subseteq
b $; and $g$ is called \emph{antiadditive} if $g(a\cup b)=g(a)\cap g(b)$.
(See, \emph{e.g.}, \cite[\S 6.1]{11ck}.)

\begin{theorem}
\label{goodness} Let $\mathcal{U}$ be a countably incomplete ultrafilter
over the infinite set $I$. Then the following are equivalent:

\begin{enumerate}
\item $\mathcal{U}$ is good.

\item In the ultrapower model of nonstandard analysis determined by $%
\mathcal{U}$, the $\kappa^+$-\emph{saturation} property holds, where $%
\kappa=|I|$.
\end{enumerate}
\end{theorem}

\begin{proof}
$(1)\Rightarrow(2)$.\ Let $\{A_x\mid x\in I\}$ be a family of internal
subsets of an hyper-extension ${}^*Y$ with the FIP. For every $x\in I$, pick
a function $\varphi_x:I\to\mathcal{P}(Y)$ such that $A_x=[\varphi_x]_%
\mathcal{U}$. By countable incompleteness, we can fix a sequence of sets 
\begin{equation*}
I_1\supseteq I_2\supseteq\ldots\supseteq I_n\supseteq I_{n+1}\supseteq\ldots
\end{equation*}
such that $I_n\in\mathcal{U}$ for all $n$ and $\bigcap_n I_n=\emptyset$.
Given $a\in\text{Fin}(I)$ of cardinality $n$, define $f(a)=\{i\in I_n\mid
\bigcap_{x\in a}\varphi_x(i)\ne\emptyset\}$. By the hypothesis, the finite
intersection $\bigcap_{x\in a}A_x\ne\emptyset$, and so $f(a)\in\mathcal{U}$.
Since $f$ is antimonotonic, by the goodness property of $\mathcal{U}$, we
can pick an antiadditive $g:\text{Fin}(I)\to\mathcal{U}$ such that $%
g(a)\subseteq f(a)$ for all $a\in\text{Fin}(I)$. For $j\in I$, let $%
\Lambda_j=\{x\mid j\in g(\{x\})\}$. If there exist distinct elements $%
y_1,\ldots,y_k\in\Lambda_j$, then 
\begin{equation*}
j\in g(\{y_1\})\cap\ldots\cap g(\{y_k\})=g(\{y_1,\ldots,y_k\})\subseteq
f(\{y_1,\ldots,y_k\})\subseteq I_k.
\end{equation*}
By recalling that $\bigcap_k I_k=\emptyset$, we can conclude that every $%
\Lambda_j$ must be finite. Notice that $j\in\bigcap_{x\in\Lambda_j}g(\{x%
\})=g(\Lambda_j)\subseteq f(\Lambda_j)$ and so we can pick an element $%
h(j)\in\bigcap_{x\in\Lambda_j}\varphi_x(j)\ne\emptyset$. In this way, for
every $x\in I$, we have $j\in g(\{x\})\Leftrightarrow
x\in\Lambda_j\Rightarrow h(j)\in\varphi_x(j)$. If $\xi=[h]_\mathcal{U}\in{}%
^*Y$ is the element in the ultrapower model that corresponds to the function 
$h:I\to Y$, then $\xi\in\bigcap_{x\in I}A_x$, because for every $x$ one has
the inclusion $\{j\in I\mid h(j)\in\varphi_x(j)\}\supseteq g(\{x\})\in%
\mathcal{U}$.

$(2)\Rightarrow(1)$.\ Let $f:\text{Fin}(I)\to\mathcal{U}$ be antimonotonic.
For $a\in\text{Fin}(I)$, let $G_a:I\to\mathcal{P}(\text{Fin}(I))$ be the
function where $G_a(i)=\{b\supseteq a\mid i\in f(b)\}$. Notice that $%
G_a(i)\cap G_b(i)=G_{a\cup b}(i)$. For every $x\in I$, $A_x=[G_{\{x\}}]_%
\mathcal{U}$ is an internal subset of ${}^*\text{Fin}(I)$, and the family $%
\{A_x\mid x\in I\}$ has the FIP because ${}^*a\in\bigcap_{x\in a}A_x$ for
every $a\in\text{Fin}(I)$. Indeed, $a\in\bigcap_{x\in
a}G_{\{x\}}(i)=G_a(i)\Leftrightarrow i\in f(a)$, and $f(a)\in\mathcal{U}$.
By $\kappa^+$-\emph{saturation}, we can pick a function $\vartheta:I\to\text{%
Fin}(I)$ such that the corresponding element in the ultrapower model $%
\xi=[\vartheta]_\mathcal{U}\in\bigcap_{x\in I}A_x$. Finally, define $%
g(a)=\{i\in I\mid\vartheta(i)\in G_a(i)\}$. Clearly, $g(a)\in\mathcal{U}$
because $\xi\in\bigcap_{x\in a}A_x$. Moreover, $i\in g(a)\Leftrightarrow
\vartheta(i)\in G_a(i)\Leftrightarrow \vartheta(i)\supseteq a$ and $i\in
f(\vartheta(i))$, and so $i\in f(\vartheta(i))\subseteq f(a)$. This shows
that $g(a)\subseteq f(a)$. Besides, $g$ is antiadditive because $i\in
g(a)\cap g(b)\Leftrightarrow \vartheta(i)\in G_a(i)$ and $\vartheta(i)\in
G_b(i)\Leftrightarrow \vartheta(i)\in G_{a\cup b}(i)\Leftrightarrow i\in
g(a\cup b)$.
\end{proof}

Points that generate good ultrafilters are characterized as follows.

\begin{theorem}
\label{goodness2} Given a model of nonstandard analysis, let $\alpha\in{}^*I$
be such that $\mathfrak{U}_\alpha$ is countably incomplete. Then the
following are equivalent:

\begin{enumerate}
\item $\mathfrak{U}_\alpha$ is a good ultrafilter\,;

\item If $F:I\to\mathcal{P}(\text{Fin}(I))$ has the property that for every
finite $a\subset I$ there exists $A\in{}^*F(\alpha)$ with $^*a\subseteq A$,
then there exists a function $\vartheta:I\to\text{Fin}(I)$ such that ${}^*x%
\in{}^*\vartheta(\alpha)\in{}^*F(\alpha)$ for all $x\in I$.
\end{enumerate}
\end{theorem}

\begin{proof}
The proof is similar to the one of the previous proposition.

$(1)\Rightarrow(2)$. Fix a sequence $I_1\supseteq
I_2\supseteq\ldots\supseteq I_n\supseteq I_{n+1}\supseteq\ldots$ of sets in $%
\mathfrak{U}_\alpha$ such that $\bigcap_n I_n=\emptyset$. Now let $F$
satisfy the conditions in $(2)$. For every finite $a=\{x_1,\ldots,x_n\}%
\subset I$ with $n$-many elements, put: 
\begin{equation*}
f(a)\ =\ \{i\in I_n\mid \exists b\in F(i)\ \text{such that}\ a\subseteq b\}.
\end{equation*}
By the hypothesis, there exists $A\in{}^*F(\alpha)$ with $%
{}^*x_1,\ldots,{}^*x_n\in A$. But then $\alpha$ is an element of ${}^*I_n$
such that ${}^*a\subseteq A$ for a suitable $A\in{}^*F(\alpha)$, and this
means that $\alpha\in{}^*(f(a))$, \emph{i.e.} $f(a)\in\mathfrak{U}_\alpha$.
Moreover, it directly follows from the definition that $f:\text{Fin}(I)\to%
\mathfrak{U}_\alpha$ is antimonotonic. So, we can apply the hypothesis and
pick an antiadditive function $g:\text{Fin}(I)\to\mathfrak{U}_\alpha$ such
that $g(a)\subseteq f(a)$ for all $a$. For $i\in I$, put $\Lambda(i)=\{x\mid
i\in g(\{x\})\}$. Notice that if $a\subseteq\Lambda(i)$ has cardinality $n$,
then $i\in\bigcap_{x\in a}g(\{x\})=g(a)\subseteq f(a)\subseteq I_n$; and
since $\bigcap I_n=\emptyset$, this shows that $\Lambda(i)$ must be finite.
Finally, pick any function $\vartheta:I\to\text{Fin}(I)$ such that $%
\Lambda(i)\subseteq\vartheta(i)\in F(i)$ for all $i\in f(\Lambda(i))$. Since 
$\alpha\in{}^*(f(\Lambda(i))$, it is ${}^*\Lambda(\alpha)\subseteq{}%
^*\vartheta(\alpha)\in{}^*F(\alpha)$. Besides, for every $x\in I$ we have
that $\alpha\in{}^*(g(\{x\}))={}^*\!g(\{{}^*x\})$, so ${}^*x\in{}%
^*\Lambda(\alpha)$, and hence ${}^*x\in{}^*\vartheta(\alpha)$.

$(2)\Rightarrow(1)$.\ Given an antimonotonic function $f:\text{Fin}(I)\to%
\mathfrak{U}_\alpha$, define $F:I\to\mathcal{P}(\text{Fin}(I))$ by putting $%
F(i)=\{a\in\text{Fin}(I)\mid i\in f(a)\}$. For every $a=\{x_1,\ldots,x_n\}\in%
\text{Fin}(I)$, $\alpha\in{}^*(f(a))={}^*\{i\in I\mid a\in F(i)\}$, and so $%
{}^*a=\{{}^*x_1,\ldots,{}^*x_n\}\in{}^*F(\alpha)$. Then, by the hypothesis
there exists a function $\vartheta:I\to\text{Fin}(I)$ such that $%
{}^*\vartheta(\alpha)\in{}^*F(\alpha)$ and ${}^*x\in{}^*\vartheta(\alpha)$
for every $x\in I$. Finally, put $g(a)=\{i\in I\mid
a\subseteq\vartheta(i)\in F(i)\}$. Notice that $\alpha\in{}^*(g(a))$,
because ${}^*a\subseteq{}^*\vartheta(\alpha)\in{}^*F(\alpha)$. It is readily
verified that $g$ is antiadditive. Moreover, notice that $\vartheta(i)\in
F(i)\Leftrightarrow i\inf(\vartheta(i))$, so $i\in g(a)\Rightarrow
a\subseteq\vartheta(i)$ and $i\in f(\vartheta(i))\Rightarrow i\in f(a)$, and
also the desired inclusion $g(a)\subseteq f(a)$ is verified.
\end{proof}

\begin{remark}
A much simplied proof of the above theorem could also be obtained by
employing a known property from nonstandard set theory. Precisely, let $%
\mathfrak{M}$ be a given model of nonstandard analysis, and let $\alpha\in{}%
^*I$ be such that the generated ultrafilter $\mathfrak{U}_\alpha$ is
countably incomplete. Then the model of nonstandard analysis determined by $%
\mathfrak{U}_\alpha$ is isomorphic to the elementary submodel $\mathfrak{M}%
[\alpha]\prec\mathfrak{M}$ whose universe is given by the elements that are 
\emph{standard relative to $\alpha$}, \emph{i.e.} by the elements of the
form ${}^*\!f(\alpha)$ where $f$ is any function defined on $I$ (see \cite[%
\S 6]{11hr} and references therein). By working inside $\mathfrak{M}[\alpha]$%
, one can directly use the equivalence of Theorem \ref{goodness} to derive
Theorem \ref{goodness2}.
\end{remark}

\section{Ultrafilters generated by pairs}

As in Section \ref{sec-uequivalence}, we now work in a fixed $\mathfrak{c}^+$%
-\emph{saturated} model of nonstandard analysis, and extend the $u$%
-equivalence to pairs.

\begin{definition}
The \emph{ultrafilter generated}%
\index{ultrafilter!generated} by an ordered pair $(\alpha,\beta)%
\in{}^*\mathbb{N}\times{}^*\mathbb{N}$ is the family 
\begin{equation*}
\mathfrak{U}_{(\alpha,\beta)}\ =\ \{X\subseteq\mathbb{N}\times\mathbb{N}%
\mid(\alpha,\beta)\in{}^*\!X\}.
\end{equation*}
\end{definition}

The $\mathfrak{c}^+$-\emph{saturation} property guarantees that every
ultrafilter on $\mathbb{N}\times\mathbb{N}$ is generated by some pair $%
(\alpha,\beta)\in{}^*\mathbb{N}\times{}^*\mathbb{N}$. The $u$-equivalence ${%
\,{\sim}_{{}_{\!\!\!\!\! u}}\;}$ relation on ${}^*\mathbb{N}\times{}^*%
\mathbb{N}$ is defined in exactly in the same way as it was defined on ${}^*%
\mathbb{N}$, that is we set $(\alpha,\beta){\,{\sim}_{{}_{\!\!\!\!\! u}}\;}%
(\alpha^{\prime },\beta^{\prime })$ when $\mathfrak{U}_{(\alpha,\beta)}=%
\mathfrak{U}_{(\alpha^{\prime },\beta^{\prime })}$. We recall that the \emph{%
Cartesian product of filters} 
\begin{equation*}
\mathcal{U}\times\mathcal{V}\ =\ \{A\times B\mid A\in\mathcal{U}, B\in%
\mathcal{V}\}
\end{equation*}
is a filter; however, if $\mathcal{U}$ and $\mathcal{V}$ are non-principal
ultrafilters, then $\mathcal{U}\times\mathcal{V}$ is \emph{not} an
ultrafilter, and indeed there may be plenty of ultrafilters $\mathcal{W}%
\supset\mathcal{U}\times\mathcal{V}$ (see Remark \ref{manyextensions} below).

A canonical class of ultrafilters on the Cartesian product is given by the
so-called tensor products: If $\mathcal{U}$ and $\mathcal{V}$ are
ultrafilters on $\mathbb{N}$, the \emph{tensor product} $\mathcal{U}\otimes%
\mathcal{V}$ is the ultrafilter on $\mathbb{N}\times\mathbb{N}$ defined by
setting: 
\begin{equation*}
X\in\mathcal{U}\otimes\mathcal{V}\ \Longleftrightarrow\ \{n\mid\{m\mid
(n,m)\in X\}\in\mathcal{V}\}\in\mathcal{U}.
\end{equation*}

It is easily seen that $\mathcal{U}\otimes\mathcal{V}\supseteq\mathcal{U}%
\times\mathcal{V}$. We recall that tensor products are not commutative in
all non-trivial cases; indeed, if we denote by $\Delta^+=\{(n,m)\mid n<m\}$
then $\Delta^+\in \mathcal{U}\otimes\mathcal{V}$ and $\Delta^+\notin\mathcal{%
V}\otimes\mathcal{U}$ for all non-principal $\mathcal{U}$ and $\mathcal{V}$.

A first easy observation about pairs is the following.

\begin{proposition}
\label{pairsandproducts} $\mathfrak{U}_{(\alpha,\beta)}\supseteq\mathfrak{U}%
_\alpha\times\mathfrak{U}_\beta$.
\end{proposition}

\begin{proof}
If $A\in\mathfrak{U}_\alpha$ and $B\in\mathfrak{U}_\beta$, then trivially $%
(\alpha,\beta)\in{}^*\!A\times{}^*B={}^*(A\times B)$, \emph{i.e.}, $A\times
B\in\mathfrak{U}_{(\alpha,\beta)}$.
\end{proof}

It is also easy to improve on the above property, and characterize those
products of ultrafilters that are contained in a given ultrafilter on $%
\mathbb{N}\times\mathbb{N}$.

\begin{proposition}
\label{productofultrafilters} $\mathfrak{U}_{(\alpha,\beta)}\supseteq%
\mathfrak{U}_\gamma\times\mathfrak{U}_\delta$ if and only if $\alpha{\,{\sim}%
_{{}_{\!\!\!\!\! u}}\;}\gamma$ and $\beta{\,{\sim}_{{}_{\!\!\!\!\! u}}\;}%
\delta$. In consequence, $(\alpha,\beta){\,{\sim}_{{}_{\!\!\!\!\! u}}\;}%
(\gamma,\delta)\Rightarrow \alpha{\,{\sim}_{{}_{\!\!\!\!\! u}}\;}\gamma$ and 
$\beta{\,{\sim}_{{}_{\!\!\!\!\! u}}\;}\delta$.
\end{proposition}

\begin{proof}
If $\alpha{\,{\sim}_{{}_{\!\!\!\!\! u}}\;}\gamma$ and $\beta{\,{\sim}%
_{{}_{\!\!\!\!\! u}}\;}\delta$, then $\mathfrak{U}_{(\alpha,\beta)}\supseteq%
\mathfrak{U}_\alpha\times\mathfrak{U}_\beta=\mathfrak{U}_\gamma\times%
\mathfrak{U}_\delta$. Conversely, assume that $\alpha{\,{\not\sim}%
_{{}_{\!\!\!\! u}}\,}\gamma$ or $\beta{\,{\not\sim}_{{}_{\!\!\!\! u}}\,}%
\delta$, say $\alpha{\,{\not\sim}_{{}_{\!\!\!\! u}}\,}\gamma$, and pick $%
A\subseteq\mathbb{N}$ such that $\alpha\in{}^*\!A$ and $\gamma\notin{}^*\!A$%
. Then $A^c\times\mathbb{N}\in\mathfrak{U}_\gamma\times\mathfrak{U}_\delta$
because $\gamma\in{}^*\!A^c$ and trivially $\delta\in{}^*\mathbb{N}$, but $%
A^c\times\mathbb{N}\notin\mathfrak{U}_{(\alpha,\beta)}$ because $%
(\alpha,\beta)\notin{}^*(A^c\times\mathbb{N})$. Finally, if $(\alpha,\beta){%
\,{\sim}_{{}_{\!\!\!\!\! u}}\;}(\gamma,\delta)$ then $\mathfrak{U}%
_{(\alpha,\beta)}=\mathfrak{U}_{(\gamma,\delta)}\supseteq\mathfrak{U}%
_\gamma\times\mathfrak{U}_\delta$, and so by what just proved above, it must
be $\alpha{\,{\sim}_{{}_{\!\!\!\!\! u}}\;}\gamma$ and $\beta{\,{\sim}%
_{{}_{\!\!\!\!\! u}}\;}\delta$.
\end{proof}

The above proposition can be reformulated in ``standard" terms as follows:

\begin{proposition}
Let $\mathcal{U},\mathcal{V}$ be ultrafilters on $\mathbb{N}$ and let $%
\mathcal{W}$ be an ultrafilter on $\mathbb{N}\times\mathbb{N}$. Then $%
\mathcal{W}\supseteq\mathcal{U}\times\mathcal{V}$ if and only if $\mathcal{U}%
=\pi_1(\mathcal{W})$ and $\mathcal{V}=\pi_2(\mathcal{W})$ where $\pi_1$ and $%
\pi_2$ are the canonical projections on the first and second coordinate,
respectively.
\end{proposition}

\begin{proof}
Pick $\alpha,\beta,\gamma,\delta\in{}^*\mathbb{N}$ such that $\mathcal{W}=%
\mathfrak{U}_{(\alpha,\beta)}$, $\mathcal{U}=\mathfrak{U}_\gamma$ and $%
\mathcal{V}=\mathfrak{U}_\delta$. By Proposition \ref{ueq1}, $\pi_1(\mathcal{%
W})=\mathfrak{U}_{\,{}^*\!\pi_1(\alpha,\beta)}=\mathfrak{U}_\alpha$ and
similarly $\pi_2(\mathcal{W})=\mathfrak{U}_\beta$. Then apply the previous
proposition.
\end{proof}

\begin{remark}
\label{manyextensions} The implication $(\alpha,\beta){\,{\sim}%
_{{}_{\!\!\!\!\! u}}\;}(\gamma,\delta)\Rightarrow \alpha{\,{\sim}%
_{{}_{\!\!\!\!\! u}}\;}\gamma$ and $\beta{\,{\sim}_{{}_{\!\!\!\!\! u}}\;}%
\delta$ cannot be reversed. Indeed, it is well known that for every
non-principal ultrafilters $\mathcal{U},\mathcal{V}$ on $\mathbb{N}$ there
are at least two different ultrafilters $\mathcal{W}\supset\mathcal{U}\times%
\mathcal{V}$, namely $\mathcal{U}\otimes\mathcal{V}$ and $\sigma(\mathcal{V}%
\otimes\mathcal{U})$ where $\sigma(n,m)=(m,n)$ is the map that permutes the
coordinates. Actually, provided there are no P-points RK-below $\mathcal{U}$
or $\mathcal{V}$, there exist infinitely many $\mathcal{W}\supset\mathcal{U}%
\times\mathcal{V}$ (see \cite{11bm} and references therein).
\end{remark}

About the existence of pairs of hypernatural numbers that generate
ultrafilters $\mathcal{W}\supseteq\mathcal{U}\times\mathcal{V}$, the
following result holds.

\begin{proposition}
\label{karel} Let $\mathcal{U},\mathcal{V}$ be ultrafilters on $\mathbb{N}$,
and let $X\in{}^*\mathcal{U}$ and $Y\in{}^*\mathcal{V}$. Then for every
ultrafilter $\mathcal{W}\supset\mathcal{U}\times\mathcal{V}$, there exist $%
\alpha\in X$ and $\beta\in Y$ such that $\mathfrak{U}_{(\alpha,\beta)}=%
\mathcal{W}$.
\end{proposition}

\begin{proof}
Since $\mathcal{W}\supset\mathcal{U}\times\mathcal{V}$, the intersection $%
(U\times V)\cap W\ne\emptyset$ is non-empty for every $U\in\mathcal{U}$, $%
V\in\mathcal{V}$, and $W\in\mathcal{W}$. Then, by \emph{transfer}, $(X\times
Y)\cap Z\ne\emptyset$ for all $Z\in{}^*\mathcal{W}$. In particular, the
family $\{(X\times Y)\cap{}^*W\mid W\in\mathcal{W}\}$ has the finite
intersection property. By $\mathfrak{c}^+$-\emph{saturation}, we can pick a
pair $(\alpha,\beta)\in X\times Y$ such that $(\alpha,\beta)\in{}^*W$ for
all $W\in\mathcal{W}$, so that $\mathfrak{U}_{(\alpha,\beta)}=\mathcal{W}$.
(This proof was communicated to the author by Karel Hrb\'a\u{c}ek, and it is
reproduced here under his permission.)
\end{proof}

Now let us fix some useful notation. For $X\subseteq\mathbb{N}\times\mathbb{N%
}$, $n\in\mathbb{N}$, and $\xi\in{}^*\mathbb{N}$:

\begin{itemize}
\item $X_n=\{\,m\in\mathbb{N}\mid (n,m)\in X\}$ is the vertical $n$-fiber of 
$X$;

\item ${}^*\!X_\xi=\{\zeta\in{}^*\mathbb{N}\mid(\xi,\zeta)\in{}^*\!X\}$ is
the vertical $\xi$-fiber of ${}^*\!X$.
\end{itemize}

Notice that ${}^*\!X_\xi={}^*\chi(\xi)$ is the value taken at $\xi$ by the
hyper-extension of the sequence $\chi(n)=X_n$. Notice also that for finite $%
k\in\mathbb{N}$, one has ${}^*\!X_k={}^*(X_k)$.

It directly follows from the definitions that 
\begin{equation*}
X\in\mathfrak{U}_\alpha\otimes\mathfrak{U}_\beta\ \Longleftrightarrow\ \alpha%
\in{}^*\{n\mid X_n\in\mathfrak{U}_\beta\}\ \Longleftrightarrow\
{}^*\!X_\alpha\in{}^*\mathfrak{U}_\beta.
\end{equation*}

\begin{theorem}
\label{alltensorpairs} Let $\alpha,\beta\in{}^*\mathbb{N}$ be infinite
numbers. Then the following properties are equivalent:

\begin{enumerate}
\item $\mathfrak{U}_{(\alpha,\beta)}=\mathfrak{U}_\alpha\otimes\mathfrak{U}%
_\beta$;

\item $(\alpha,\beta)$ generates a tensor product;

\item For every $F:\mathbb{N}\to\mathcal{P}(\mathbb{N})$, if $\beta\in{}%
^{*}\!F (\alpha)$, then ${}^{*}\!F (\alpha)\in{}^*\mathfrak{U}_\beta$;

\item For every $F:\mathbb{N}\to\mathcal{P}(\mathbb{N})$, if ${}^{*}\!F
(\alpha)\in{}^*\mathfrak{U}_\beta$, then $\beta\in{}^{*}\!F (\alpha)$;

\item For every $X\subseteq\mathbb{N}\times\mathbb{N}$, if $(n,\beta)\in{}%
^*\!X$ for all $n\in\mathbb{N}$, then $(\alpha,\beta)\in{}^*\!X$;

\item For every $X\subseteq\mathbb{N}\times\mathbb{N}$, if $(\alpha,\beta)%
\in{}^*\!X$, then $(n,\beta)\in{}^*\!X$ for some $n\in\mathbb{N}$;

\item For every $f:\mathbb{N}\to\mathbb{N}$, if ${}^*\!f(\beta)\notin\mathbb{%
N}$, then ${}^*\!f(\beta)>\alpha$.
\end{enumerate}
\end{theorem}

\begin{proof}
$(1)\Leftrightarrow(2)$. One implication is trivial. Conversely, let us
assume that $\mathfrak{U}_{(\alpha,\beta)}=\mathfrak{U}_\gamma\otimes%
\mathfrak{U}_\delta$ for some $\gamma,\delta$. We have seen in Proposition %
\ref{pairsandproducts} that $\mathfrak{U}_{(\alpha,\beta)}$ extends $%
\mathfrak{U}_\alpha\times\mathfrak{U}_\beta$; moreover, the tensor product $%
\mathfrak{U}_\gamma\otimes\mathfrak{U}_\delta$ extends $\mathfrak{U}%
_\gamma\times\mathfrak{U}_\delta$. But then it must be $\mathfrak{U}_\alpha=%
\mathfrak{U}_\gamma$ and $\mathfrak{U}_\beta=\mathfrak{U}_\delta$, as
otherwise there would be disjoint sets in $\mathfrak{U}_{(\alpha,\beta)}$.

$(1)\Leftrightarrow(3)\Leftrightarrow(4)$. Given a function $F:\mathbb{N}\to%
\mathcal{P}(\mathbb{N})$, let 
\begin{equation*}
\Theta(F)\ =\ \{(n,m)\in\mathbb{N}\times\mathbb{N}\mid m\in F(n)\}
\end{equation*}
be the set of pairs whose vertical $n$-fibers $\Theta(F)_n=F(n)$. Then, we
have $\Theta(F)\in\mathfrak{U}_\alpha\otimes\mathfrak{U}_\beta%
\Leftrightarrow {}^*\Theta(F)_\alpha={}^{*}\!F (\alpha)\in{}^*\mathfrak{U}%
_\beta$. Moreover, $\beta\in{}^{*}\!F (\alpha)\Leftrightarrow (\alpha,\beta)%
\in{}^*\Theta(F)\Leftrightarrow \Theta(F)\in\mathfrak{U}_{(\alpha,\beta)}$.
Now notice that for every $X\subseteq\mathbb{N}\times\mathbb{N}$, one has $%
X=\Theta(F)$ where $F(n)=X_n$ is the sequence of the $n$-fibers of $X$. In
consequence, properties $(3)$ and $(4)$ are equivalent to the conditions $%
\mathfrak{U}_{(\alpha,\beta)}\subseteq\mathfrak{U}_\alpha\otimes\mathfrak{U}%
_\beta$ and $\mathfrak{U}_\alpha\otimes\mathfrak{U}_\beta\subseteq\mathfrak{U%
}_{(\alpha,\beta)}$, respectively. The proof is complete because inclusions
between ultrafilters imply equalities.

$(5)\Leftrightarrow(6)$. This is straightforward, because property $(5)$ for
a set $X$ is the contrapositive of property $(6)$ for the complement $X^c$.

$(1)\Rightarrow(5)$. Notice that $(n,\beta)\in{}^*\!X \Leftrightarrow \beta%
\in{}^*(X_n)\Leftrightarrow X_n\in\mathfrak{U}_\beta$. So, by the
hypothesis, the set $\{n\mid X_n\in\mathfrak{U}_\beta\}=\mathbb{N}$. As
trivially $\mathbb{N}\in\mathfrak{U}_\alpha$, we conclude that $X\in%
\mathfrak{U}_\alpha\otimes\mathfrak{U}_\beta=\mathfrak{U}_{(\alpha,\beta)}$,
and hence $(\alpha,\beta)\in{}^*\!X$.

$(5)\Rightarrow(7)$. Let $X=\{(n,m)\mid n<f(m)\}$. Since ${}^*\!f(\beta)$ is
infinite, we have that $(n,\beta)\in{}^*\!X$ for all finite $n\in\mathbb{N}$%
. But then $(\alpha,\beta)\in{}^*\!X$, \emph{i.e.} $\alpha<{}^*\!f(\beta)$.

$(7)\Rightarrow(1)$ (\cite{11pu2}, Th.\,3.4). The set $\Delta^+=\{(n,m)\mid
n<m\}$ belongs to $\mathfrak{U}_\alpha\otimes\mathfrak{U}_\beta$; moreover,
since $\alpha<\beta$, it is also $\Delta^+\in\mathfrak{U}_{(\alpha,\beta)}$.
Thus, it suffices to show that the implication $X\in\mathfrak{U}%
_\alpha\otimes\mathfrak{U}_\beta\Rightarrow X\in\mathfrak{U}%
_{(\alpha,\beta)} $ holds for all $X\subseteq\Delta^+$. By definition, $X\in%
\mathfrak{U}_\alpha\otimes\mathfrak{U}_\beta\Leftrightarrow \alpha\in{}%
^*\{n\mid X_n\in\mathfrak{U}_\beta\}$, and so $X_n\in\mathfrak{U}_\beta$ for
infinitely many $n$. In consequence, for every $m$ one can always find $n>m$%
, and hence $(n,m)\notin X$, such that $X_n\in\mathfrak{U}_\beta$. This
means that 
\begin{equation*}
F(m)=\{n\mid X_n\in\mathfrak{U}_\beta\ \&\ (n,m)\notin X\}\ne\emptyset.
\end{equation*}
If $f:\mathbb{N}\to\mathbb{N}$ is the function $f(m)=\min F(m)$, then the
number ${}^*\!f(\beta)$ is infinite. Indeed, if by contradiction $%
{}^*\!f(\beta)=k\in\mathbb{N}$, then we would have $k\in {}^{*}\!F (\beta)$,
that is $({}^*\!X)_k={}^*(X_k)\in{}^*\mathfrak{U}_\beta$ and $(k,\beta)%
\notin{}^*\!X$, and hence $X_k\in\mathfrak{U}_\beta$ and $X_k\notin\mathfrak{%
U}_\beta$, a contradiction. Now, by $(7)$ it follows that $%
\alpha<{}^*\!f(\beta)=\min {}^{*}\!F (\beta)$, and so $\alpha\notin{}^{*}\!F
(\beta)$. This means that it is not the case that both ${}^*\!X_\alpha\in{}^*%
\mathfrak{U}_\beta$ and $(\alpha,\beta)\notin{}^*\!X$. We reach the thesis $%
(\alpha,\beta)\in{}^*\!X$ by recalling that ${}^*\!X_\alpha\in{}^*\mathfrak{U%
}_\beta\Leftrightarrow X\in\mathfrak{U}_\alpha\otimes\mathfrak{U}_\beta$.
\end{proof}

\begin{definition}
We say that $(\alpha,\beta)\in{}^*\mathbb{N}\times{}^*\mathbb{N}$ is a \emph{%
tensor pair}%
\index{tensor pair} if it satisfies all the equivalent conditions in the
previous theorem.
\end{definition}

Notice that if $n\in\mathbb{N}$ is finite, then all pairs $(n,\beta)$ are
trivially tensor pairs. The property of generating tensor products is
preserved under a special class of maps.

\begin{proposition}
\label{imagetensor} If $(\alpha,\beta)$ is a tensor pair then for every $f,g:%
\mathbb{N}\to\mathbb{N}$ also $({}^*\!f(\alpha),{}^*\!g(\beta))$ is a tensor
pair. In ``standard" terms, this means that if $(f,g)$ is the map $%
(n,m)\mapsto(f(n),g(m))$, then for all ultrafilters $\mathcal{U}$ and $%
\mathcal{V}$ the image $(f,g)(\mathcal{U}\otimes\mathcal{V})=f(\mathcal{U}%
)\otimes g(\mathcal{V})$.
\end{proposition}

\begin{proof}
We use the characterization of tensor pairs as given by (6) of Proposition %
\ref{alltensorpairs}. Notice that ${}^*(f,g)=({}^*\!f,{}^*\!g)$, and so 
\begin{equation*}
({}^*\!f(\alpha),{}^*\!g(\beta))%
\in{}^*\!X\ \Leftrightarrow\
(\alpha,\beta)\in({}^*\!f,{}^*\!g)^{-1}({}^*\!X)\ =\ {}^*[(f,g)^{-1}(X)].
\end{equation*}
By the hypothesis, we can pick $n\in\mathbb{N}$ such that $(n,\beta)\in{}%
^*[(f,g)^{-1}(X)]$, that is $({}^*\!f(n),{}^*\!g(\beta))\in{}^*\!X$, where $%
{}^*\!f(n)=f(n)\in\mathbb{N}$.
\end{proof}

As we already noticed in Remark \ref{manyextensions}, both the tensor
product $\mathcal{U}\otimes\mathcal{U}$ and its image $\sigma(\mathcal{U}%
\otimes\mathcal{U})$ under the map $\sigma(n,m)=(m,n)$ extend the Cartesian
product $\mathcal{U}\times\mathcal{U}$. However, there is another canonical
ultrafilter that extend the product of an ultrafilter $\mathcal{U}$ by
itself, namely the \emph{diagonal ultrafilter} determined by $\mathcal{U}$: 
\begin{equation*}
\Delta_\mathcal{U}\ =\ \{X\subseteq \mathbb{N}\times\mathbb{N}\mid\{n\mid
(n,n)\in X\}\in\mathcal{U}\}.
\end{equation*}
Clearly $\Delta_\mathcal{U}\cong\mathcal{U}$. Notice that the diagonal $%
\Delta=\{(n,n)\mid n\in\mathbb{N}\}\in\Delta_\mathcal{U}$, but $\Delta\notin%
\mathcal{U}\otimes\mathcal{U}$ and $\Delta\notin \sigma(\mathcal{U}\otimes%
\mathcal{U})$ whenever $\mathcal{U}$ is non-principal. Notice also that if $%
\mathcal{U}=\mathfrak{U}_\alpha$, then $\Delta_\mathcal{U}=\mathfrak{U}%
_{(\alpha,\alpha)}$.

\begin{proposition}
Let $\mathcal{U}$ be non-principal, and let $\mathcal{W}$ be an ultrafilter
that extends the product filter $\mathcal{U}\times\mathcal{U}$. Then $%
\mathcal{W}\le_{RK}\mathcal{U}$ if and only if $\mathcal{W}=\Delta_\mathcal{U%
}$.
\end{proposition}

\begin{proof}
Let $\mathcal{U}=\mathfrak{U}_\alpha$. If $h$ is the function such that $%
h(n)=(n,n)$ for all $n$, then $h(\mathcal{U})=\mathfrak{U}_{{}^*\!h(\alpha)}=%
\mathfrak{U}_{(\alpha,\alpha)}=\Delta_\mathcal{U}$, and so $\Delta_\mathcal{U%
}\le_{RK}\mathcal{U}$. Conversely, let $\mathcal{W}\supset\mathcal{U}\times%
\mathcal{U}$ and assume that $F(\mathfrak{U}_\alpha)=\mathcal{W}$ for a
suitable function $F:\mathbb{N}\to\mathbb{N}\times\mathbb{N}$, say $%
F(n)=(f(n),g(n))$. By Proposition \ref{productofultrafilters}, we can pick $%
\beta{\,{\sim}_{{}_{\!\!\!\!\! u}}\;}\gamma{\,{\sim}_{{}_{\!\!\!\!\! u}}\;}%
\alpha$ such that $\mathcal{W}=\mathfrak{U}_{(\beta,\gamma)}$. Then 
\begin{equation*}
\mathfrak{U}_{(\beta,\gamma)}=F(\mathfrak{U}_\alpha)=\mathfrak{U}_{{}^{*}\!F
(\alpha)}=\mathfrak{U}_{({}^*\!f(\alpha),{}^*\!g(\alpha))}.
\end{equation*}
Thus, ${}^*\!f(\alpha){\,{\sim}_{{}_{\!\!\!\!\! u}}\;}\beta{\,{\sim}%
_{{}_{\!\!\!\!\! u}}\;}\alpha$ and ${}^*\!g(\alpha){\,{\sim}_{{}_{\!\!\!\!\!
u}}\;}\gamma{\,{\sim}_{{}_{\!\!\!\!\! u}}\;}\alpha$ and so, by Theorem \ref%
{u-identity}, ${}^*\!f(\alpha)={}^*\!g(\alpha)=\alpha$. We conclude that $%
\mathcal{W}=\mathfrak{U}_{(\alpha,\alpha)}=\Delta_\mathcal{U}$.
\end{proof}

\begin{corollary}
If $\mathcal{U}$ is a non-principal ultrafilter, then $\mathcal{U}\otimes%
\mathcal{U}\not\le_{RK}\mathcal{U}$ (and hence $\mathcal{U}\otimes\mathcal{U}%
\not\cong\mathcal{U}$).
\end{corollary}

We close this section by showing that tensor pairs are found in abundance.

\begin{theorem}
\label{existencetensorpairs} Let $\alpha,\beta\in{}^*\mathbb{N}$ be
infinite. Then:

\begin{enumerate}
\item The set $R_{\alpha,\beta}=\{\beta^{\prime }\mid\beta^{\prime }{\,{\sim}%
_{{}_{\!\!\!\!\! u}}\;}\beta\ \&\ (\alpha,\beta^{\prime })\ \text{tensor pair%
}\}$ contains $|{}^*\mathbb{N}|$-many elements and it is unbounded in ${}^*%
\mathbb{N}$\,;

\item The set $L_{\alpha,\beta}=\{\alpha^{\prime }\mid\alpha^{\prime }{\,{%
\sim}_{{}_{\!\!\!\!\! u}}\;}\alpha\ \&\ (\alpha^{\prime },\beta)\ \text{%
tensor pair}\}$ is unbounded leftward in ${}^*\mathbb{N}\setminus\mathbb{N}$%
, and hence it contains at least $\mathfrak{c}^+$-many elements.
\end{enumerate}
\end{theorem}

\begin{proof}
$(1)$. We prove the thesis by showing that there exists an internal
hyperinfinite set of elements $\beta^{\prime }{\,{\sim}_{{}_{\!\!\!\!\! u}}\;%
}\beta$ such that $(\alpha,\beta^{\prime })$ is a tensor pair. Let $%
B\subseteq\mathbb{N}$ and $f:\mathbb{N}\to\mathbb{N}$ be such that $\beta\in{%
}^*B$ and ${}^*\!f(\beta)\notin\mathbb{N}$. Then ${}^*\!f(\beta)\in{}%
^*(f(B)) $ implies that the image set $f(B)$ is infinite, and so the
following property holds: 
\begin{equation*}
\forall y\in\mathbb{N}\ \exists\,\sigma:\mathbb{N}\to B\ \text{1-1}\ \text{%
s.t.}\ \forall x\in\mathbb{N}\ f(\sigma(x))>y.
\end{equation*}
By \emph{transfer}, it follows that the following internal set is non-empty: 
\begin{equation*}
\Gamma_{B,f}\ =\ \{\sigma:{}^*\mathbb{N}\to{}^*B\ \text{internal}\mid
\sigma\ \text{1-1}\ \&\ {}^*\!f({}^*\sigma(\nu))>\alpha\ \text{for all }\nu%
\in{}^*\mathbb{N}\}.
\end{equation*}
Notice that $\Gamma_{B_1,f_1}\cap\ldots\cap\Gamma_{B_k,f_k}=\Gamma_{B,f}$
where $B=B_1\cap\ldots\cap B_k$ and $f=\min\{f_1,\ldots,f_k\}$. Then, the
family $\{\Gamma_{B,f}\mid \beta\in{}^*B\,\&\,{}^*\!f(\beta)\notin\mathbb{N}%
\}$ has the finite intersection property and so, by $\mathfrak{c}^+$%
-enlarging, we can pick an element 
\begin{equation*}
\tau\ \in\ \bigcap\left\{{}^*\Gamma_{B,f}\mid \beta\in{}^*B\ \&\
{}^*\!f(\beta)\notin\mathbb{N}\right\}.
\end{equation*}
Notice that $\tau(\nu){\,{\sim}_{{}_{\!\!\!\!\! u}}\;}\beta$ for all $\nu\in{%
}^*\mathbb{N}$, and therefore $\text{range}(\tau)$ is an hyperinfinite set
of elements $\beta^{\prime }{\,{\sim}_{{}_{\!\!\!\!\! u}}\;}\beta$ such that 
${}^*\!f(\beta^{\prime })>\alpha$ whenever ${}^*\!f(\beta)\notin\mathbb{N}$.
Finally, since $\beta^{\prime }{\,{\sim}_{{}_{\!\!\!\!\! u}}\;}\beta$, one
has that ${}^*\!f(\beta)\notin\mathbb{N}\Leftrightarrow{}^*\!f(\beta^{\prime
})\notin\mathbb{N}$; so, condition $(7)$ of Theorem \ref{alltensorpairs}
applies, and all pairs $(\alpha,\beta^{\prime })$ are tensor pairs.

$(2)$. We recall that by $\mathfrak{c}^+$-\emph{saturation}, the
coinitiality of ${}^*\mathbb{N}\setminus\mathbb{N}$ is greater than $%
\mathfrak{c}$, and so we can pick an infinite $\xi$ such that $%
\xi<{}^*\!f(\beta)$ for all $f$ with ${}^*\!f(\beta)\notin\mathbb{N}$. Then,
by $(7)$ of Theorem \ref{alltensorpairs}, $(\alpha^{\prime },\beta)$ is a
tensor pair for all $\alpha^{\prime }<\xi$, and the thesis follows because
the set $u(\alpha)\cap[0,\xi)$ is unbounded leftward in ${}^*\mathbb{N}%
\setminus\mathbb{N}$, by Theorem \ref{monads}.
\end{proof}

\section{Hyper-shifts}

The following notion was introduced by M. Beiglb\"ock in \cite{11be}:

\begin{definition}
Let $A\subseteq\mathbb{N}$ and let $\mathcal{U}$ be an ultrafilter on $%
\mathbb{N}$. The \emph{ultrafilter-shift} of $A$ by $\mathcal{U}$ is defined
as the set 
\begin{equation*}
A-\mathcal{U}\ =\ \{n\in\mathbb{N}\mid A-n\in\mathcal{U}\}.
\end{equation*}
\end{definition}

We now introduce a class of subsets of $\mathbb{N}$, which are found as
segments of hyper-extensions, and that precisely correspond to
ultrafilter-shifts.

\begin{definition}
The \emph{hyper-shift}%
\index{hyper-shift} of a set $A\subseteq\mathbb{N}$ by a number $\gamma%
\in{}^*\mathbb{N}$ is the following set: 
\begin{equation*}
A_\gamma\ =\ ({}^*\!A-\gamma)\cap\mathbb{N}\ =\ \{n\in\mathbb{N}\mid\gamma+n%
\in{}^*\!A\}.
\end{equation*}
\end{definition}

It is readily seen that hyper-shifts are coherent with respect to finite
shifts, intersections, unions, and complements.

\begin{proposition}
For every $A,B\subseteq\mathbb{N}$, for every $n\in\mathbb{N}$, and for
every $\gamma\in{}^*\mathbb{N}$, the following equalities hold:

\begin{enumerate}
\item $(A-n)_\gamma=A_\gamma-n$\,;

\item $(A\cap B)_\gamma=A_\gamma\cap B_\gamma$\,;

\item $(A\cup B)_\gamma=A_\gamma\cup B_\gamma$\,;

\item $(A^c)_\gamma=(A_\gamma)^c$.
\end{enumerate}
\end{proposition}

\begin{proposition}
\label{ultrashifts} For every $A\subseteq\mathbb{N}$ and for every $\gamma%
\in{}^*\mathbb{N}$, 
\begin{equation*}
A_\gamma\ =\ A-\mathfrak{U}_\gamma.
\end{equation*}
\end{proposition}

\begin{proof}
The following chain of equivalences is directly obtained from the
definitions: $n\in A_\gamma\Leftrightarrow \gamma+n\in{}^*\!A\Leftrightarrow
\gamma\in {}^*\!A-n={}^*(A-n)\Leftrightarrow A-n\in\mathfrak{U}%
_\gamma\Leftrightarrow n\in A-\mathfrak{U}_\gamma$.
\end{proof}

Ultrafilter-shifts (and their nonstandard counterparts, the hyper-shifts)
have a precise combinatorial meaning corresponding to a notion of
embeddability.

\begin{definition}
Let $A,B\subseteq\mathbb{N}$. We say that $A$ is \emph{exactly embedded} in $%
B$, and write $A\le_e B$, if for every finite interval $I$ there exists $x$ such
that $x+(A\cap I)=B\cap(x+I)$.
\end{definition}

A similar relation between sets of natural numbers, named ``finite
embeddability", has been considered in \cite[\S 4]{11dn} (see also \cite%
{11bd}, where that notion was extended to ultrafilters). The difference is
that ``$A$ finitely embedded in $B$" only requires the inclusion $x+(A\cap
I)\subseteq B\cap(x+I)$.

With respect to finite configurations, $A\le_e B$ tells that $B$ is at least
as combinatorially rich as $A$. For example, if $A$ contains arbitrarily
long arithmetic progressions and $A\le_e B$, then also $B$ contains
arbitrarily long arithmetic progressions.

\begin{proposition}
For $A,B\subseteq\mathbb{N}$, the following are equivalent:

\begin{enumerate}
\item $A\le_e B$\.;

\item $A$ is an ultrafilter-shift of $B$\,;

\item $A=B_\gamma$ for some $\gamma\in{}^*\mathbb{N}$.
\end{enumerate}
\end{proposition}

\begin{proof}
$(2)\Leftrightarrow(3)$ is given by Proposition \ref{ultrashifts}.

$(1)\Rightarrow(3)$. Given $n\in\mathbb{N}$, let us consider the sets 
\begin{equation*}
\Lambda_n\ =\ \left\{x\in\mathbb{N}\mid x+(A\cap[1,n])=B\cap[x+1,x+n]%
\right\}.
\end{equation*}
By the hypothesis, $\Lambda_n\ne\emptyset$ for all $n\in\mathbb{N}$ and so,
by \emph{overspill}, we can pick $\gamma\in{}^*\Lambda_\nu$ for some
infinite $\nu\in{}^*\mathbb{N}$. (We denoted by ${}^*\Lambda_\nu={}^{*}\!F
(\nu)$ the value taken at $\nu$ by the hyper-extension of the function $%
F(n)=\Lambda_n$.) Then $\gamma+({}^*\!A\cap[1,\nu])={}^*B\cap[%
\gamma+1,\gamma+\nu]$, which implies $A=({}^*B-\gamma)\cap\mathbb{N}%
=B_\gamma$.

$(3)\Rightarrow(1)$.\ If $A=B_\gamma$, then for every interval $I\subset%
\mathbb{N}$ we have ${}^*\!A\cap I=A\cap I=B_\gamma\cap I=({}^*B-\gamma)\cap
I$. So, $\gamma$ is a witness of the following property: ``$\exists \xi\in{}%
^*\mathbb{N},\ \xi+({}^*\!A\cap I)={}^*B\cap(\xi+I)$." By \emph{transfer} we
obtain the thesis: ``$\exists x\in\mathbb{N},\ x+(A\cap I)=B\cap (x+I)$."
\end{proof}

Let us now see how ultrafilter-shifts and hyper-shifts can be used to
characterize density. Recall the following

\begin{definition}
The \emph{Schnirelmann density} of $A\subseteq\mathbb{N}$ is defined as 
\begin{equation*}
\sigma(A)\ =\ \inf_{n\in\mathbb{N}}\frac{|A\cap[1,n]|}{n}.
\end{equation*}

The \emph{asymptotic density} of $A\subseteq\mathbb{N}$ is defined as 
\begin{equation*}
d(A)\ =\ \lim_{n\to\infty}\frac{|A\cap[1,n]|}{n},
\end{equation*}
when the limit exists. Otherwise, one defines the \emph{upper density} $%
\overline{d}(A)$ and the \emph{lower density} $\underline{d}(A)$ by taking
the limit superior or the limit inferior of the above sequence, respectively.
\end{definition}

Another notion that is used in combinatorial number theory is the following
uniform version of the asymptotic density.

\begin{definition}
The \emph{Banach density}%
\index{Banach density} of $A\subseteq\mathbb{N}$ is defined as 
\begin{equation*}
\text{BD}(A)\ =\ \lim_{n\to\infty}\frac{\max_{k\in\mathbb{N}}|A\cap[k+1,k+n]|%
}{n}.
\end{equation*}
\end{definition}

Notice that the sequence $a_n=\max_k|A\cap[k+1,k+n]|$ is \emph{subadditive}, 
\emph{i.e.} it satisfies $a_{n+m}\le a_n+a_m$. In consequence, one can show
that $\lim_{n}a_n/n$ actually exists, and in fact $\lim_{n}a_n/n=%
\inf_{n}a_n/n$. It is readily checked that for every $A\subseteq\mathbb{N}$: 
\begin{equation*}
\sigma(A)\ \le\ \underline{d}(A)\ \le\ \overline{d}(A)\ \le\ \text{BD}(A).
\end{equation*}
On the other hand, there exist sets where all the above inequalities are
strict.

Positive Banach densities are preserved under exact embeddings (but the same
property does not hold neither for Schnirelmann nor for asymptotic
densities).

\begin{proposition}
If $B\le_e A$ then $\text{BD}(B)\le\text{BD}(A)$. In consequence:

\begin{enumerate}
\item $\text{BD}(A-\mathcal{U})\le\text{BD}(A)$ for all ultrafilters $%
\mathcal{U}$ on $\mathbb{N}$\,;

\item $\text{BD}(A_\gamma)\le\text{BD}(A)$ for all $\gamma\in{}^*\mathbb{N}$.
\end{enumerate}
\end{proposition}

\begin{proof}
Let $\langle I_n\mid n\in\mathbb{N}\rangle$ be a sequence of intervals with
length $|I_n|=n$ and such that $\lim_n |B\cap I_n|/n=\text{BD}(B)$. By the
hypothesis, for every $n$ there exists $x_n$ such that $x_n+(B\cap
I_n)=A\cap J_n$ where $J_n=x_n+I_n$ is the interval of length $n$ obtained
by shifting $I_n$ by $x_n$. Then $\text{BD}(A)\ge\lim_n|A\cap J_n|/n=
\lim_n|B\cap I_n|/n=\text{BD}(B)$.
\end{proof}

The Banach density of a set equals the maximum density of its
ultrafilter-shifts.

\begin{theorem}
For every $A\subseteq\mathbb{N}$ there exists a hyper-shift $A_\gamma$ such
that 
\begin{equation*}
\sigma(A_\gamma)\ =\ d(A_\gamma)\ =\ \text{BD}(A_\gamma)\ =\ \text{BD}(A).
\end{equation*}
Equivalently, there exists an ultrafilter $\mathcal{U}=\mathfrak{U}_\gamma$
on $\mathbb{N}$ such that 
\begin{equation*}
\sigma(A-\mathcal{U})\ =\ d(A-\mathcal{U})\ =\ \text{BD}(A-\mathcal{U})\ =\ 
\text{BD}(A).
\end{equation*}
\end{theorem}

\begin{proof}
By the nonstandard characterization of limit, given any infinite $N$, we can
pick an interval $I=[\Omega+1,\Omega+N]\subset{}^*\mathbb{N}$ of length $N$
such that $\|{}^*\!A\cap I\|/N=a\approx\text{BD}(A)$. (We use delimiters $%
\|X\|$ to denote the internal cardinality of an internal set $X$, and $%
\xi\approx\eta$ to mean that $\xi$ and $\eta$ are infinitely close, that is, 
$\xi-\eta$ is infinitesimal.) Now fix an infinite $\nu$ such that $%
\nu/N\approx 0$.

\begin{itemize}
\item \emph{Claim.}\ There exists $\gamma$ such that for every $1\le i\le\nu$%
: 
\begin{equation*}
\frac{\|{}^*\!A\cap[\gamma+1,\gamma+i]\|}{i}\ \ge\ a-\frac{\nu}{N}.
\end{equation*}
\end{itemize}

Notice that the above claim yields the thesis. Indeed, for every finite $n\in%
\mathbb{N}$, trivially $n\le\nu$, and so 
\begin{equation*}
\frac{|A_\gamma\cap[1,n]|}{n}\ =\ \frac{\|{}^*\!A\cap[\gamma+1,\gamma+n]\|}{n%
}\ \ge\ a-\frac{\nu}{N}\ \approx\ \text{BD}(A).
\end{equation*}
This implies that $\sigma(A_\gamma)\ge\text{BD}(A)$, and the thesis follows
because $\sigma(A_\gamma)\le\underline{d}(A_\gamma)\le\overline{d}%
(A_\gamma)\le \text{BD}(A_\gamma)\le\text{BD}(A)$. To prove the claim, let
us proceed by contradiction, and for every $\gamma$, let 
\begin{equation*}
\psi(\gamma)\ =\ \min\left\{ 1\le i\le\nu \ \Big|\ \frac{\|{}^*\!A\cap[%
\gamma+1,\gamma+i]\|}{i}<a-\frac{\nu}{N}\right\}.
\end{equation*}
By internal induction, define $\xi_1=\Omega$, $\xi_{j+1}=\xi_j+\psi(\xi_j)$,
and stop at step $\mu$ when $\Omega+N-\nu\le\xi_\mu<\Omega+N$. Then we would
have 
\begin{equation*}
a-\frac{\nu}{N}\ \le\ \frac{\|{}^*\!A\cap[\Omega+1,\xi_\mu]\|}{N}\ =\ \frac{1%
}{N}\sum_{j=1}^{\mu-1}\|{}^*\!A\cap[\xi_j+1,\xi_{j+1}]\|\ <
\end{equation*}
\begin{equation*}
<\ \frac{1}{N}\sum_{j=1}^{\mu-1}\psi(\xi_j)\left(a-\frac{\nu}{N}\right)\ =\ 
\frac{\xi_\mu-\xi_1}{N}\left(a-\frac{\nu}{N}\right)\ <\ a-\frac{\nu}{N},
\end{equation*}
that gives the desired contradiction.
\end{proof}

A series of relevant results in additive combinatorics have been recently
obtained by R. Jin by nonstandard analysis. In some of them, nonstandard
properties of Banach density like the one stated in above theorem, play an
important role (see, \emph{e.g.}, the survey \cite{11jin} and references
therein).

\section{Nonstandard characterizations in the space $(\protect\beta\mathbb{N}%
,\oplus)$}

The space of ultrafilters $\beta\mathbb{N}$ can be equipped with
a``pseudo-sum" operation $\oplus$ that extends the usual sum between natural
numbers. The resulting space $(\beta\mathbb{N},\oplus)$ and its
generalizations have been extensively studied during the last decades,
producing plenty of interesting applications in Ramsey theory and
combinatorics of numbers (see the monography \cite{11hs}).

\begin{definition}
The \emph{pseudo-sum}%
\index{ultrafilter!pseudo-sum} $\mathcal{U}\oplus\mathcal{V}$ of two
ultrafilters $\mathcal{U},\mathcal{V}$ on $\mathbb{N}$ is the image $S(%
\mathcal{U}\otimes\mathcal{V})$ of the tensor product $\mathcal{U}\otimes%
\mathcal{V}$ under the sum function $S(n,m)=n+m$. Equivalently, for every $%
A\subseteq\mathbb{N}$: 
\begin{equation*}
A\in\mathcal{U}\oplus\mathcal{V}\ \Longleftrightarrow\ \{n\mid A-n\in%
\mathcal{V}\}\in\mathcal{U},
\end{equation*}
where $A-n=\{m\in\mathbb{N}\mid m+n\in A\}$ is the \emph{leftward $n$-shift}
of $A$.
\end{definition}

Notice that $\oplus$ actually extends the sum on $\mathbb{N}$; indeed, if $%
\mathcal{U}_n$ and $\mathcal{U}_m$ are the principal ultrafilters determined
by $n$ and $m$ respectively, then $\mathcal{U}_n\oplus\mathcal{U}_m=\mathcal{%
U}_{m+n}$ is the principal ultrafilter determined by $n+m$.

The following property is a straight consequence of the definitions.

\begin{proposition}
\label{pseudosumgenerators} $\mathcal{U}\oplus\mathcal{V}=\mathcal{W}$ if
and only if there exists a tensor pair $(\alpha,\beta)$ such that $\mathfrak{%
U}_\alpha=\mathcal{U}$, $\mathfrak{U}_\beta=\mathcal{V}$ and $\mathfrak{U}%
_{\alpha+\beta}=\mathcal{W}$. In consequence: 
\begin{equation*}
\{\mathcal{U}\oplus\mathcal{V}\mid \mathcal{U},\mathcal{V}\in\beta\mathbb{N}%
\}\ =\ \{\mathfrak{U}_{\alpha+\beta}\mid (\alpha,\beta)\ 
\text{tensor pair\,}\}.
\end{equation*}
\end{proposition}

\begin{proof}
Given $\mathcal{W}=\mathfrak{U}_\xi\oplus\mathfrak{U}_\eta$, pick a pair $%
(\alpha,\beta)$ with $\mathfrak{U}_{(\alpha,\beta)}=\mathfrak{U}_\xi\otimes%
\mathfrak{U}_\eta$. Then $(\alpha,\beta)$ is a tensor pair such that $%
\mathfrak{U}_\alpha=\mathfrak{U}_\xi$, $\mathfrak{U}_\beta=\mathfrak{U}_\eta$%
, and $\mathcal{W}=S(\mathfrak{U}_\xi\otimes\mathfrak{U}_\eta)= S(\mathfrak{U%
}_{(\alpha,\beta)})=\mathfrak{U}_{\alpha+\beta}$. Conversely, if $%
(\alpha,\beta)$ is a tensor pair, then $\mathfrak{U}_\alpha\oplus\mathfrak{U}%
_\beta=S(\mathfrak{U}_\alpha\otimes\mathfrak{U}_\beta)= S(\mathfrak{U}%
_{(\alpha,\beta)})=\mathfrak{U}_{\alpha+\beta}$.
\end{proof}

We recall that the space $\beta\mathbb{N}$ of ultrafilters on $\mathbb{N}$
is endowed with the topology as determined by the following family of basic
(cl)open sets, for $A\subseteq\mathbb{N}$: 
\begin{equation*}
\mathcal{O}_A\ =\ \{\mathcal{U}\mid A\in\mathcal{U}\}.
\end{equation*}

In this way, $\beta\mathbb{N}$ is the \emph{Stone-\v{C}ech compactification}
of the discrete space $\mathbb{N}$. Indeed, $\beta\mathbb{N}$ is Hausdorff
and compact, and if one identifies each $n\in\mathbb{N}$ with the
corresponding principal ultrafilter, then $\mathbb{N}$ is dense in $\beta%
\mathbb{N}$. Moreover, one can prove that every $f:\mathbb{N}\to K$ where $K$
is Hausdorff compact space $K$ has a (unique) continuous extension $\beta
f:\beta\mathbb{N}\to K$. When equipped with the pseudo-sum operation, $\beta%
\mathbb{N}$ has the structure of a \emph{compact topological left semigroup}%
, because for any fixed $\mathcal{V}$, the ``product on the left'' by $%
\mathcal{U}$: 
\begin{equation*}
\psi_\mathcal{V}:\mathcal{U}\mapsto\mathcal{U}\oplus\mathcal{V}
\end{equation*}
is a continuous function. We remark that the pseudo-sum operation is
associative, but it fails badly to be commutative (see Proposition \ref%
{noncommutativity}).

Connections between pseudo-sums and hyper-shifts are established in the next
proposition.

\begin{proposition}
\label{hyper-shifts} Let $\alpha,\beta,\gamma\in{}^*\mathbb{N}$. Then:

\begin{enumerate}
\item $A\in\mathfrak{U}_\alpha\oplus\mathfrak{U}_\beta\ \Leftrightarrow\
A_\beta\in\mathfrak{U}_\alpha\ \Leftrightarrow\ \alpha\in{}^*(A_\beta)$.

\item For every $n\in\mathbb{N}$, $(A-n)_\beta=A_\beta-n$.

\item For every $n\in\mathbb{N}$, $A-n\in\mathfrak{U}_\alpha\oplus\mathfrak{U%
}_\beta\ \Leftrightarrow\ n\in(A_\beta)_\alpha$.

\item $\mathfrak{U}_\gamma=\mathfrak{U}_\alpha\oplus\mathfrak{U}_\beta\
\Leftrightarrow\ A_\gamma=(A_\beta)_\alpha$ for every $A\subseteq\mathbb{N}$.

\item If $(\alpha,\beta)$ is a tensor pair, then $A_{\alpha+\beta}=(A_%
\beta)_\alpha$ for every $A\subseteq\mathbb{N}$.
\end{enumerate}
\end{proposition}

\begin{proof}
Equivalences in $(1)$ directly follow from the definitions. $(2)$ is proved
by the chain of equivalences: $k\in (A-n)_\beta\Leftrightarrow k+\beta\in{}%
^*(A-n)={}^*\!A-n\Leftrightarrow k+\beta+n\in{}^*\!A\Leftrightarrow k+n\in
A_\beta\Leftrightarrow k\in A_\beta-n$. $(3)$ By using the previous
properties, we obtain $A-n\in\mathfrak{U}_\alpha\oplus\mathfrak{U}%
_\beta\Leftrightarrow \alpha\in{}^*[(A-n)_\beta]={}^*(A_\beta-n)={}^*(A_%
\beta)-n\Leftrightarrow \alpha+n\in{}^*\!A_\beta\Leftrightarrow
n\in(A_\beta)_\alpha$. $(4)$ Assume first $\mathfrak{U}_\gamma=\mathfrak{U}%
_\alpha\oplus\mathfrak{U}_\beta$. For every $n\in\mathbb{N}$, by $(3)$ one
has that $n\in(A_\beta)_\alpha\Leftrightarrow A-n\in\mathfrak{U}_\alpha\oplus%
\mathfrak{U}_\beta \Leftrightarrow A-n\in\mathfrak{U}_\gamma\Leftrightarrow%
\gamma\in{}^*(A-n)={}^*\!A-n \Leftrightarrow n\in A_\gamma$, and so $%
(A_\beta)_\alpha=A_\gamma$. Conversely, assume by contradiction that we can
pick $A\in\mathfrak{U}_\alpha\oplus\mathfrak{U}_\beta$ with $A\notin%
\mathfrak{U}_\gamma$. Then $\alpha\in{}^*(A_\beta)$ and $\gamma\notin{}^*\!A$%
, and hence $(A_\beta)_\alpha\ne A_\gamma$ because $0\in(A_\beta)_\alpha$
but $0\notin A_\gamma$. $(5)$ By Proposition \ref{pseudosumgenerators}, $%
\mathfrak{U}_\alpha\oplus\mathfrak{U}_\beta=\mathfrak{U}_{\alpha+\beta}$,
and so the thesis directly follows from $(4)$.
\end{proof}

As a first example of use of hyper-shifts in $(\beta\mathbb{N},\oplus)$, let
us prove the continuity of the ``product on the left" functions. This is
easily done by showing that $\psi^{-1}_{\mathfrak{U}_\beta}(\mathcal{O}_A)=%
\mathcal{O}_{A_\beta}$ for every $\beta\in{}^*\mathbb{N}$ and for every $%
A\subseteq\mathbb{N}$; indeed, for every $\alpha\in{}^*\mathbb{N}$ one has: 
\begin{equation*}
\mathfrak{U}_\alpha\in\psi_{\mathfrak{U}_\beta}^{-1}(\mathcal{O}_A)\
\Leftrightarrow\ \mathfrak{U}_\alpha\oplus\mathfrak{U}_\beta\in\mathcal{O}%
_A\ \Leftrightarrow
\end{equation*}
\begin{equation*}
A\in\mathfrak{U}_\alpha\oplus\mathfrak{U}_\beta\ \Leftrightarrow\ A_\beta\in%
\mathfrak{U}_\alpha\ \Leftrightarrow\ \mathfrak{U}_\alpha\in\mathcal{O}%
_{A_\beta}.
\end{equation*}

As one can readily verify, $\mathcal{U}_n\oplus\mathcal{V}=\mathcal{V}\oplus%
\mathcal{U}_n$ for every principal ultrafilter $\mathcal{U}_n$. We now use
hyper-shifts to show that the center of $(\beta\mathbb{N},\oplus)$ actually
contains only principal ultrafilters.

\begin{theorem}
\label{noncommutativity} For every non-principal ultrafilter $\mathcal{U}$
there exists a non-principal ultrafilter $\mathcal{V}$ such that $\mathcal{U}%
\oplus\mathcal{V}\ne\mathcal{V}\oplus\mathcal{U}$.
\end{theorem}

\begin{proof}
Pick an infinite $\gamma$ such that $\mathcal{U}=\mathfrak{U}_\gamma$, and
let $\nu$ be such that $\nu^2\le\gamma<(\nu+1)^2$. Let us assume that $\nu$
is even (the case $\nu$ odd is entirely similar), and let 
\begin{equation*}
A\ =\ \bigcup_{n\ \text{even}}[n^2,(n+1)^2).
\end{equation*}
We distinguish two cases. If $(\nu+1)^2-\gamma$ is infinite, let $%
\beta=(\nu+1)^2$. Notice that $A_\gamma=\mathbb{N}$ because $\gamma+n\in{}%
^*\!A$ for all $n$, and $A_\beta=\emptyset$ because $\beta+n\notin{}^*\!A$
for all $n$. Then $A\in\mathfrak{U}_\beta\oplus\mathfrak{U}_\gamma$ because
trivially $A_\gamma\in\mathfrak{U}_\beta$, and $A\notin\mathfrak{U}%
_\gamma\oplus\mathfrak{U}_\beta$ because trivially $A_\beta\notin\mathfrak{U}%
_\gamma$. If $(\nu+1)^2-\gamma$ is finite, let $\beta=\nu^2$. In this case $%
A_\gamma$ is finite, and $A_\beta=\mathbb{N}$. Then $A\notin\mathfrak{U}%
_\beta\oplus\mathfrak{U}_\gamma$ because $A_\gamma\notin\mathfrak{U}_\beta$,
and $A\in\mathfrak{U}_\gamma\oplus\mathfrak{U}_\beta$ because $A_\beta\in%
\mathfrak{U}_\gamma$.
\end{proof}

\section{Idempotent ultrafilters}

Ultrafilters that are idempotent with respect to pseudo-sums play an
instrumental role in applications.

\begin{definition}
An ultrafilter $\mathcal{U}$ on $\mathbb{N}$ is called \emph{idempotent}%
\index{ultrafilter!idempotent} if $\mathcal{U}\oplus\mathcal{U}=\mathcal{U}$%
, \emph{i.e.} if 
\begin{equation*}
A\in\mathcal{U}\ \Longleftrightarrow\ \{n\mid A-n\in\mathcal{U}\}\in\mathcal{%
U}.
\end{equation*}
\end{definition}

\begin{theorem}
\label{idemequivalences} Let $\alpha%
\in{}^*\mathbb{N}$. The following properties are equivalent:

\begin{enumerate}
\item $\mathfrak{U}_\alpha$ is idempotent\,;

\item There exists a tensor pair $(\alpha,\beta)$ such that $\alpha{\,{\sim}%
_{{}_{\!\!\!\!\! u}}\;}\beta{\,{\sim}_{{}_{\!\!\!\!\! u}}\;}\alpha+\beta$\,;

\item For every $A\subseteq\mathbb{N}$, $A_\alpha=(A_\alpha)_\alpha$\,;

\item For every $A\subseteq\mathbb{N}$, $(A\cap A_\alpha)_\alpha=A_\alpha$\,;

\item If $\alpha\in{}^*\!A$ then $\alpha\in{}^*(A_\alpha)$\,;

\item If $\alpha\in{}^*\!A$ then $A\cap A_\alpha\ne\emptyset$\,;

\item If $\alpha\in{}^*\!A$ then there exists $B\subseteq A\cap B_\alpha$
such that $\alpha\in{}^*B$.

\item For every $A\in\mathfrak{U}_\alpha$ there exists $a\in A$ such that $%
A-a\in\mathfrak{U}_\alpha$\,;

\item For every $A\in\mathfrak{U}_\alpha$ there exists $B\subseteq A$ such
that $B\in\mathfrak{U}_\alpha$ and $B-b\in\mathfrak{U}_\alpha$ for all $b\in
B$.
\end{enumerate}
\end{theorem}

\begin{proof}
$(1)\Leftrightarrow(2)$. By Theorem \ref{existencetensorpairs}, we can pick $%
\beta$ such that $(\alpha,\beta)$ is a tensor pair and $(\alpha,\beta){\,{%
\sim}_{{}_{\!\!\!\!\! u}}\;}(\alpha,\alpha)$. If $\mathfrak{U}_\alpha$ is
idempotent, then 
\begin{equation*}
\mathfrak{U}_\beta\ =\ \mathfrak{U}_\alpha\ =\ \mathfrak{U}_\alpha\oplus%
\mathfrak{U}_\alpha\ =\ S(\mathfrak{U}_\alpha\otimes\mathfrak{U}_\alpha)\ =\
S(\mathfrak{U}_{(\alpha,\beta)})\ =\ \mathfrak{U}_{\alpha+\beta},
\end{equation*}
where $S$ denotes the sum function. Conversely, by Proposition \ref%
{pseudosumgenerators}, 
\begin{equation*}
\mathfrak{U}_\alpha\oplus\mathfrak{U}_\alpha\ =\ \mathfrak{U}_\alpha\oplus%
\mathfrak{U}_\beta\ =\ \mathfrak{U}_{\alpha+\beta}\ =\ \mathfrak{U}_\alpha.
\end{equation*}

$(1)\Leftrightarrow(3)$ is given by property $(4)$ of Proposition \ref%
{hyper-shifts}.

$(3)\Leftrightarrow(4)$. Notice that $(A\cap
A_\alpha)_\alpha=A_\alpha\cap(A_\alpha)_\alpha$, so property $(4)$ is
equivalent to $A_\alpha\subseteq(A_\alpha)_\alpha$ for every $A\subseteq%
\mathbb{N}$, and one implication is trivial. The converse implication $%
(A_\alpha)_\alpha\subseteq A_\alpha$ is proved by considering complements: 
\begin{equation*}
(A_\alpha)^c\ =\ (A^c)_\alpha\ \subseteq\ [(A^c)_\alpha]_\alpha\ =\
[(A_\alpha)^c]_\alpha\ =\ [(A_\alpha)_\alpha]^c.
\end{equation*}

$(1)\Leftrightarrow(5)$. We recall that $\alpha\in{}^*(A_\alpha)%
\Leftrightarrow A\in\mathfrak{U}_\alpha\oplus\mathfrak{U}_\alpha$. So $(5)$
states the inclusion between ultrafilters $\mathfrak{U}_\alpha\subseteq%
\mathfrak{U}_\alpha\oplus\mathfrak{U}_\alpha$, which is equivalent to
equality $\mathfrak{U}_\alpha=\mathfrak{U}_\alpha\oplus\mathfrak{U}_\alpha$.

$(5)\Rightarrow(6)$. Since $\alpha\in{}^*\!A\cap{}^*(A_\alpha)={}^*(A\cap
A_\alpha)$, by \emph{transfer} we obtain the thesis.

$(6)\Rightarrow(5)$. Assume by contradiction that $\alpha\in{}^*\!A$ but $%
\alpha\notin{}^*(A_\alpha)$. Then $\alpha\in{}^*B$ where $%
B=A\cap(A_\alpha)^c $. This is against the hypothesis because $B\cap
B_\alpha=(A\cap (A_\alpha)^c)\cap (A_\alpha\cap[(A_\alpha)_\alpha]^c)=
\emptyset$.

$(3)\,\&\,(5)\Rightarrow(7)$. Let $B=A\cap A_\alpha$. Then, $\alpha\in{}^*\!A%
\cap{}^*(A_\alpha)={}^*B$. Moreover, trivially $B\subseteq A$, and also 
\begin{equation*}
B\ \subseteq\ A_\alpha\ =\ A_\alpha\cap(A_\alpha)_\alpha\ =\ (A\cap
A_\alpha)_\alpha\ =\ B_\alpha.
\end{equation*}

$(7)\Rightarrow(6)$. Given $\alpha\in{}^*\!A$, pick $B$ as in the
hypothesis. Notice that $B\subseteq A$; moreover $B\subseteq
B_\alpha\subseteq(A\cap B_\alpha)_\alpha=A_\alpha\cap
(B_\alpha)_\alpha\subseteq A_\alpha$, and hence $B\subseteq A\cap A_\alpha$.
We conclude by noticing that $\alpha\in{}^*B$, and so, by \emph{transfer}, $%
B\ne\emptyset$.

Finally, it is easily verified that properties $(8)$ and $(9)$ are the
``standard'' counterparts of properties $(6)$ and $(7)$, respectively.
\end{proof}

We remark that the existence of idempotent ultrafilters is not a trivial
matter: it is proved as an application of a general result by R. Ellis about
the existence of idempotents in every compact Hausdorff topological left
semigroup (see, \emph{e.g.}, \cite{11hs}). Indeed, $(\beta\mathbb{N},\oplus)$
is a compact Hausdorff topological left semigroup.

Historically, the first application of idempotent ultrafilters in
combinatorics of numbers was a short and elegant proof of Hindman's theorem
found by F. Galvin and S. Glazer. The original argument used by N. Hindman
in his proof \cite{11ht} was really intricated. Actually, Hindman himself
wrote in \cite{11hj}: \emph{``If the reader has a graduate student that she
wants to punish, she should make him read and understand that original proof"%
}. A detailed report about the discovery of the ultrafilter proof can be
found in \cite{11hj}.

\begin{theorem}[Hindman]
\index{Hindman's Theorem} For every finite partition $\mathbb{N}%
=C_1\cup\ldots\cup C_r$ of the natural numbers, there exists an infinite set 
$X$ such that every (finite) sum of distinct elements from $X$ belongs to
the same piece $C_i$.
\end{theorem}

\begin{proof}
Pick $\alpha$ such that $\mathfrak{U}_\alpha$ is idempotent, and let $C_i$
be the piece of the partition such that $\alpha%
\in{}^*C_i$. By $(7)$ of Theorem \ref{idemequivalences}, we can fix a set $%
B\subseteq C_i\cap B_\alpha$ with $\alpha\in{}^*B$. Notice that $x\in
B\Rightarrow x\in B_\alpha\Leftrightarrow \alpha+x\in{}^*B\Leftrightarrow%
\alpha\in{}^*(B-x)$.

Now pick any $x_1\in B$. Then $\alpha$ witnesses the existence in ${}^*B$ of
elements larger than $x_1$ that belong to ${}^*(B-x_1)$. By \emph{transfer},
we obtain the existence in $B$ of an element $x_2>x_1$ such that $x_2\in
B-x_1$. Notice that $x_1,x_2,x_1+x_2\in B$, and hence $\alpha\in{}^*(B-x_2)$
and $\alpha\in{}^*(B-x_1-x_2)$. Similarly as above, $\alpha$ witnesses the
existence in ${}^*B$ of elements larger than $x_2$ that belong to $%
{}^*[(B-x_1)\cap(B-x_2)\cap(B-x_1-x_2)]$. Again by using \emph{transfer}, we
get the existence in $B$ of an element $x_3>x_2$ such that $x_3\in
(B-x_1)\cap(B-x_2)\cap(B-x_1-x_2)$, and so we have that $%
x_1,x_2,x_3,x_1+x_2,x_1+x_3,x_2+x_3,x_1+x_2+x_3\in B$. By iterating the
process, we eventually obtain a set $X=\{x_1<x_2<\ldots<x_n<\ldots\}$ such
that every sum of distinct elements from $X$ belongs to $B$, and hence to
the same piece $C_i$ of the partition, as desired.
\end{proof}

We recall that in the definition of the pseudo-sum $\mathcal{U}\oplus%
\mathcal{V}$, one considers \emph{leftward} shifts $A-n=\{m\mid m+n\in A\}$.
By considering instead \emph{rightward} shifts $A+n=\{a+n\mid a\in A\}$, one
obtains the following operation.

\begin{definition}
Let $\mathcal{U},\mathcal{V}$ be ultrafilters on $\mathbb{N}$, where $%
\mathcal{V}$ is non-principal. The ultrafilter $\mathcal{U}\star\mathcal{V}$
is defined by setting for every $A\subseteq\mathbb{N}$: 
\begin{equation*}
A\in\mathcal{U}\star\mathcal{V}\ \Leftrightarrow\ \{n\mid A+n\in\mathcal{V}%
\}\in\mathcal{U}.
\end{equation*}
\end{definition}

Notice that one can identify $\mathcal{U}\star\mathcal{V}$ with the image $D(%
\mathcal{U}\otimes\mathcal{V})$ of the tensor product $\mathcal{U}\otimes%
\mathcal{V}$ under the difference function $D(n,m)=m-n$. Indeed, although $D$
takes values in $\mathbb{Z}$, one has that $\mathbb{N}\in\mathcal{U}\star%
\mathcal{V}$ whenever $\mathcal{V}$ is non-principal and so, in this case,
one can restrict to subsets of $\mathbb{N}$.

In a similar way as done in Theorem \ref{idemequivalences}, one can prove
several nonstandard characterizations of ultrafilters that are idempotent
with respect to $\star$. In particular, corresponding to item $(2)$ in
Theorem \ref{idemequivalences}, it is shown that $\mathfrak{U}_\alpha\star%
\mathfrak{U}_\alpha=\mathfrak{U}_\alpha$ if and only if there exists a
tensor pair $(\alpha,\beta)$ such that $\alpha{\,{\sim}_{{}_{\!\!\!\!\! u}}\;%
}\beta{\,{\sim}_{{}_{\!\!\!\!\! u}}\;}\beta-\alpha$. The problem is that
there can be no such pair!

\begin{theorem}
\label{norightidem} There exist pairs $(\alpha,\beta)$ such that $\alpha{\,{%
\sim}_{{}_{\!\!\!\!\! u}}\;}\beta{\,{\sim}_{{}_{\!\!\!\!\! u}}\;}%
\beta-\alpha $, but not one of them is a tensor pair. In consequence, there
exist no ultrafilters $\mathcal{U}$ such that $\mathcal{U}\star\mathcal{U}=%
\mathcal{U} $.
\end{theorem}

\begin{proof}
For every $A\subseteq\mathbb{N}$, let us consider the set 
\begin{equation*}
\Gamma(A)\ =\ \{(a,b)\in\mathbb{N}\times\mathbb{N}\mid \text{either}\
a,b,b-a\in A\ \text{or}\ a,b,b-a\in A^c\}.
\end{equation*}
We want to show that the family $\{\Gamma(A)\mid A\subseteq\mathbb{N}\}$ has
the finite intersection property. Once this is proved, by $\mathfrak{c}^+$-%
\emph{saturation} one can pick a pair $(\alpha,\beta)\in\bigcap_{A\subseteq%
\mathbb{N}}{}^*\Gamma(A)$. It is then easily verified that $\alpha{\,{\sim}%
_{{}_{\!\!\!\!\! u}}\;}\beta{\,{\sim}_{{}_{\!\!\!\!\! u}}\;}\beta-\alpha$.

Given $A_1,\ldots,A_n$, pick a finite partition $\mathbb{N}%
=C_1\cup\ldots\cup C_r$ such that, for every piece $C_i$ and every set $A_j$%
, one has that either $C_i\subseteq A_j$ or $C_i\subseteq A_j^c$. Now, by
Rado's theorem, the equation $X-Y=Z$ is \emph{partition regular} on $\mathbb{%
N}$, and so we can pick elements $x,y,z\in C_i$ in one piece of the
partition that satisfy $x-y=z$; in consequence, $(x,y)\in\bigcap_{i=1}^r%
\Gamma(C_j)\subseteq \bigcap_{j=1}^n\Gamma(A_i)\ne\emptyset$. We recall that
an equation $f(X_1,\ldots,X_n)=0$ is called \emph{partition regular} on $%
\mathbb{N}$ when for every finite partition of $\mathbb{N}$ there exist $%
x_1,\ldots,x_n$ in the same piece of the partition that solve the equation, 
\emph{i.e.} $f(x_1,\ldots,x_n)=0$. Rado's theorem states that a linear
equation $c_1X_1+\ldots+c_nX_n=0$ (where the $c_i\ne 0$) is partition
regular on $\mathbb{N}$ if and only if there exists a sum of distinct
coefficients that equals zero (see \cite[Ch.3]{11grs}).

Let us now turn to the negative result, and assume $\alpha{\,{\sim}%
_{{}_{\!\!\!\!\! u}}\;}\beta\,{\,{\sim}_{{}_{\!\!\!\!\! u}}\;}\,\beta-\alpha$%
. Notice that both $\alpha$ and $\beta$ must be multiples of every (finite)
natural number. Indeed, given $n\ge 2$, the $u$-equivalence of $\alpha$ and $%
\beta$ implies that $\alpha\equiv \beta\mod n$, and hence $%
\beta-\alpha\equiv 0\mod n$. But $\beta-\alpha$ is $u$-equivalent to both $%
\alpha$ and $\beta$, and so $\alpha\equiv\beta\equiv\beta-\alpha\equiv 0%
\mod
n$. Now, let us consider the functions $f,g:\mathbb{N}\to\mathbb{N}\cup\{0\}$
where $f(n)$ is the greatest exponent such that $3^{f(n)}$ divides $n$, and $%
g(n)=n/3^{f(n)}$. By what proved above, both ${}^*\!f(\alpha)$ and $%
{}^*\!f(\beta)$ are infinite. Since $\alpha{\,{\sim}_{{}_{\!\!\!\!\! u}}\;}%
\beta$, we have ${}^*\!g(\alpha){\,{\sim}_{{}_{\!\!\!\!\! u}}\;}%
{}^*\!g(\beta)$, and so ${}^*\!g(\alpha)\equiv{}^*\!g(\beta)=j\mod 3$, where
either $j=1$ or $j=2$. Now assume by contradiction that $(\alpha,\beta)$ is
a tensor pair. By Proposition \ref{imagetensor}, also $({}^*\!f(\alpha),{}^*%
\!f(\beta))$ is a tensor pair, and since both components are infinite, it is 
${}^*\!f(\alpha)<{}^*\!f(\beta)$. Then we have: 
\begin{equation*}
\beta-\alpha\ =\ 3^{{}^*\!f(\beta)}\cdot{}^*\!g(\beta)\,-\,
3^{{}^*\!f(\alpha)}\cdot{}^*\!g(\alpha)\ =\ 3^{{}^*\!f(\alpha)}\cdot(3^\nu%
\cdot{}^*\!g(\beta)-{}^*\!g(\alpha))
\end{equation*}
where $\nu={}^*\!f(\beta)-{}^*\!f(\alpha)>0$. In consequence, $%
{}^*\!f(\beta-\alpha)={}^*\!f(\alpha)$ and ${}^*\!g(\beta-\alpha)=3^\nu\cdot{%
}^*\!g(\beta)-{}^*\!g(\alpha)\equiv -{}^*\!g(\alpha)\equiv -j\mod 3$, while $%
\beta-\alpha\,{\,{\sim}_{{}_{\!\!\!\!\! u}}\;}\,\alpha\Rightarrow
{}^*\!g(\beta-\alpha){\,{\sim}_{{}_{\!\!\!\!\! u}}\;}{}^*\!g(\alpha) 
\Rightarrow{}^*\!g(\beta-\alpha)\equiv{}^*\!g(\alpha)\equiv j\mod (3)$. We
must conclude that $-j\equiv j\mod 3$, and hence $j=0$, a contradiction.
(The last part of this argument is essentially the same as the one used by
Hindman in \cite[\S 4]{11h72} to prove the non-existence of ultrafilters $%
\mathcal{U}=\mathcal{U}\star\mathcal{U}$.)
\end{proof}

\section{Final remarks and open questions}

Since a first draft of this paper was written in 2009, several applications
of the presented nonstandard approach to the use of ultrafilters appeared in
the literature. In \cite{11dn-rado}, iterated nonstandard extensions were
used to characterize idempotent ultrafilters along the lines of Theorem \ref%
{idemequivalences}; and by using suitable linear combinations of idempotent
ultrafilters, a new proof of a version of Rad\'o's Theorem was given.
Partition regularity of (nonlinear) polynomial equations by nonstandard
methods is the subject-matter of the paper \cite{11lu-pr}. In \cite{11bd}, a
notion of finite embeddability between sets and between ultrafilters is
investigated, also with the use of the hyper-shifts of \S 5. The papers \cite%
{11lu-fe,11lu-fe2} continue that line of research: the nonstandard approach is
exploited to further investigating the relationships between finite
embeddability relations, algebraic properties in $(\beta\mathbb{N},\oplus)$,
and combinatorial structure of sets of natural numbers.

We like to close this paper with some remarks about idempotent ultrafilters.
To this day, basically the only known proof of their existence is grounded
on an old result by R. Ellis, namely the fact that every compact Hausdorff
topological left semigroup has idempotents (see \cite{11el}). Since
idempotent ultrafilters are widely used in applications, it seems desirable
to better understand them; to this end, a solution to the following problem
would be valuable.

\begin{itemize}
\item Open problem \#1: \emph{Find an alternative, nonstandard proof of the
existence of idempotent ultrafilters.}
\end{itemize}

Our notions of $u$-equivalence and of tensor pair, and hence of idempotent
ultrafilter, can be generalized to models $M$ of any first-order theory $%
\mathsf{T}\supseteq\mathsf{PA}$ that extends \emph{Peano Arithmetic}. (By
this we mean that $\mathsf{T}$ is a collection of sentences in a first-order
language $\mathcal{L}$ that extends the language of \textsf{PA}.)

We recall that the \emph{type} of an element $a\in M$ is the set of all
formulas with one free variable that are satisfied by $a$ in $M$: 
\begin{equation*}
tp(a)\ =\ \{\varphi(x)\mid M\models\varphi(a)\}.
\end{equation*}
Another notion that makes sense for models $M$ of theories $\mathsf{T}%
\supseteq\mathsf{PA}$ is the following. Call a pair $(a,b)\in M\times M$ 
\emph{independent} when for every formula $\varphi(x,y)$, if $%
M\models\varphi(a,b)$ then $M\models\varphi(k,b)$ for some $k\in\mathbb{N}$.
(This definition corresponds to the notion of \emph{heir of a type}, as used
in stable theories.)

If $\text{Th}(\mathbb{N})$ is the first-order theory of the natural numbers
in the full language that contains a symbol for every relation, function and
constant of $\mathbb{N}$, then $M\models\text{Th}(\mathbb{N})$ means that $%
M={}^*\mathbb{N}$ is the set of hypernatural numbers of a model of
nonstandard analysis. In this case, trivially every subset $A\subseteq%
\mathbb{N}$ is definable, and hence $tp(a)=tp(b)$ if and only if $a{\,{\sim}%
_{{}_{\!\!\!\!\! u}}\;} b$. Moreover, $(a,b)$ is independent means that $%
(a,b)$ is a tensor pair (see $(6)$ of Theorem \ref{alltensorpairs}). Thus,
by using property $(2)$ of Theorem \ref{idemequivalences}, one could propose
the following generalization of idempotent ultrafilter.

\begin{definition}
Let $M\models\mathsf{T}\supseteq\mathsf{PA}$. We say that an element $%
\alpha\in M$ is \emph{idempotent} if there exists an independent pair $%
(\alpha,\beta)$ such that $tp(\alpha)=tp(\beta)=tp(\alpha+\beta)$.
\end{definition}

\begin{itemize}
\item Open problem \# 2: \emph{Given a first-order theory $\mathsf{T}%
\supseteq\mathsf{PA}$, find sufficient conditions for models $M\models%
\mathsf{T}$ to contain idempotent elements}.
\end{itemize}

We recall that in any $\mathfrak{c}^{+}$-\emph{saturated} model of
nonstandard analysis, every ultrafilter on $\mathbb{N}$ is generated by some
element $\alpha \in {}^{\ast }\mathbb{N}$. In consequence, all $\mathfrak{c}%
^{+}$-saturated models of $\text{Th}(\mathbb{N})$ contain idempotent
elements. Isolating model-theoretic properties that guarantee the existence
of idempotent elements would probably be useful also to attack the previous
open problem.


\section{References}

\begin{thebibliography}{}
\bibitem{11nato} L.O. Arkeryd, N.J. Cutland, and C.W. Henson (eds.), \emph{%
Nonstandard Analysis -- Theory and Applications}, NATO ASI Series C 493, 
Kluwer A.P., 1997.

\bibitem{11bh} D. Ballard and K. Hrb\'a\u{c}ek, Standard foundations for
nonstandard analysis, \emph{J. Symb. Logic} \textbf{57} (1992), pp. 471--478.

\bibitem{11bs} T. Bartoszynski and S. Shelah, On the density of Hausdorff
ultrafilters, in \emph{Logic Colloquium 2004}, A. Andretta, K. Kearnes, and
D. Zambella (eds.), Lecture Notes in Logic 29, A.S.L. - Cambridge
University Press, 2008, 18--32.

\bibitem{11be} M. Beiglb\"ock, An ultrafilter approach to Jin's theorem, 
\emph{Israel J. Math.} \textbf{185}(2011), 369--374.

\bibitem{11ber} V. Bergelson, Ergodic Ramsey Theory -- an update, in \emph{%
Ergodic Theory of $\mathbb{Z}^d$-Actions} (Warwick 1993-94), London
Mathematical Society Lecture Note Ser. 228, Cambridge University
Press, 1996, 1--61.

\bibitem{11ultra} V. Bergelson, A. Blass, M. Di Nasso, and R. Jin (eds.), 
\emph{Ultrafilters across Mathematics}, Contemporary Mathematics 530, 
Amer. Math. Soc., 2010.

\bibitem{11bl} A. Blass, A model-theoretic view of some special
ultrafilters, in \emph{Logic Colloquium '77}, A. Macintyre, L. Paciolski and
J. Paris (eds.), North-Holland, 1978, 79--90.

\bibitem{11blh} A. Blass, Combinatorial characteristics of the continuum, in 
\emph{Handbook of Set Theory}, M. Foreman and A. Kanamori (eds.), Springer,
2010, 395--489.

\bibitem{11bd} A. Blass and M. Di Nasso, Finite embeddability of sets and
ultrafilters, arXiv:1405.2841, submitted.

\bibitem{11bm} A. Blass and G. Moche, Finite preimages under the natural map
from $\beta(\mathbb{N}\times\mathbb{N})$ to $\beta\mathbb{N}\times\beta%
\mathbb{N}$, \emph{Topology Proc.} \textbf{26}(2001-2002), 407--432.

\bibitem{11ck} C.C. Chang and H.J. Keisler, \emph{Model Theory} (3rd
edition), North-Holland, 1990.

\bibitem{11ch} G. Cherlin and J. Hirschfeld, Ultrafilters and ultraproducts
in non-standard analysis, in \emph{Contributions to Non-Standard Analysis},
W.A.J. Luxemburg and A. Robinson (eds.), North Holland, 1972, 261--279.

\bibitem{11cho} G. Choquet, Deux classes remarquables d'ultrafiltres sur $%
\mathbb{N}$, \emph{Bull. Sc. Math.} \textbf{92}(1968), 143--153.

\bibitem{11cn} W.W. Comfort and S. Negropontis, \emph{The Theory of
Ultrafilters}, Springer-Verlag, 1974.

\bibitem{11co} A. Connes, Ultrapuissances et applications dans le cadre de
l'analisi non standard, \emph{S\'eminaire Choquet, Initiation \`a l'analyse}%
, Tome \textbf{9} (1), 1969-1970.

\bibitem{11lnl} N.J. Cutland, M. Di Nasso, and D.A. Ross (eds.), \emph{%
Nonstandard Methods and Applications in Mathematics}, Lecture Notes in
Logic 25, A.S.L. - A.K. Peters, 2006.

\bibitem{11dn} M. Di Nasso, Embeddability properties of difference sets, 
\emph{Integers} \textbf{14}(2014),  A27.

\bibitem{11dn-rado} M. Di Nasso, Iterated hyper-extensions and an idempotent
ultrafilter proof of Rado's theorem, \emph{Proc. Amer. Math. Soc.}, to
appear.

\bibitem{11dfS} M. Di Nasso and M. Forti, Ultrafilter semirings and
nonstandard submodels of the Stone-\v{C}ech compactification of the natural
numbers, in \emph{Logic and its Applications} (A. Blass and Y. Zhang, eds.),
AMS Contemporary Mathematics 380, 2005, 45--51.

\bibitem{11dfH} M. Di Nasso and M. Forti, Hausdorff ultrafilters, \emph{%
Proc. Am. Math. Soc.} \textbf{134}(2006), 1809--1818.

\bibitem{11el} R. Ellis, \emph{Lectures on Topological Dynamics}, Benjamin,
New York, 1969.

\bibitem{11ga} D. Galvin, Ultrafilters, with applications to analysis,
social choice and combinatorics, unpublished notes, 2009.

\bibitem{11grs} R. Graham, B. Rothschild, and J. Spencer, \emph{Ramsey Theory}
(2nd edition), John Wiley \& Sons, 1990.

\bibitem{11h72} N. Hindman, The existence of certain ultrafilters on $%
\mathbb{N}$ and a conjecture of Graham and Rothschild, \emph{Proc. Am. Math.
Soc.} \textbf{36}(1972), 341--346.

\bibitem{11ht} N. Hindman, Finite sums from sequences within cells of a
partition of $\mathbb{N}$, \emph{J. Comb. Theory (Series A)} 
\textbf{17}(1974), 1--11.

\bibitem{11hj} N. Hindman, Algebra in the Stone-\v{C}ech compactification
and its applications to Ramsey Theory, \emph{Sci. Math. Jpn.} 
\textbf{62}(2005), 321--329.

\bibitem{11hs} N. Hindman and D. Strauss, \emph{Algebra in the Stone-\v{C}%
ech Compactification, Theory and Applications} (2nd edition), W. de Gruyter,
2011.

\bibitem{11hi} J. Hirshfeld, Nonstandard Combinatorics, \emph{Studia Logica} 
\textbf{47}(1988), 221--232.

\bibitem{11hr} K. Hrb\'a\u{c}ek, Realism, nonstandard set theory, and large
cardinals, \emph{Ann. Pure Appl. Logic} \textbf{109}(2001), 15--48.

\bibitem{11jin} R. Jin, Ultrapower of $\mathbb{N}$ and density problems, in 
\cite{11ultra}, 147--161.

\bibitem{11kr} V. Kanovei and M. Reeken, \emph{Nonstandard Analysis,
Axiomatically}, Springer, 2004.

\bibitem{11ke63} H.J. Keisler, Limit ultrapowers, \emph{Trans. Amer. Math.
Soc.} \textbf{107}(1963), 382--408.

\bibitem{11kt} P. Komj\'ath and V. Totik, Ultrafilters, \emph{Amer. Math.
Monthly} \textbf{115}(2009), 33--44.

\bibitem{11ku} K. Kunen, Some points in $\beta\mathbb{N}$, \emph{Math. Proc.
Cambridge Phil. Soc.} \textbf{80}(1976), 385--398.

\bibitem{11lw} P.A. Loeb and M. Wolff (eds.), \emph{Nonstandard Analysis for
the Working Mathematician}, Kluwer A.P., 2000.

\bibitem{11lu-pr} L. Luperi Baglini, Partition regularity of nonlinear
polynomials: a nonstandard approach, \emph{Integers} \textbf{14}(2014), A30.

\bibitem{11lu-fe} L. Luperi Baglini, Ultrafilters maximal for finite
embeddability, arXiv:1401.4977, submitted.

\bibitem{11lu-fe2} L. Luperi Baglini, F-finite embeddabilities of sets and
ultrafilters, arXiv:1401.6518, submitted.

\bibitem{11lux62} W.A.J. Luxemburg, \emph{Nonstandard Analysis} (Lecture
Notes), Department of Mathematics, California Institute of Technology, 1962.

\bibitem{11lux69} W.A.J. Luxemburg, A general theory of monads, in \emph{%
Applications of Model Theory to Algebra, Analysis and Probability},
W.A.J. Luxemburg (ed.), Holt, Rinehart \& Winston, 1969, 18---86.

\bibitem{11vm} J. van Mill, An introduction to $\beta\omega$, in \emph{%
Handbook of set-theoretic topology}, K. Kunen and J.E. Vaughan (eds.),
North-Holland, 1984.

\bibitem{11nr} S. Ng and H. Render, The Puritz order and its relationship to
the Rudin-Keisler order, in \emph{Reuniting the Antipodes: Constructive and
Nonstandard Views of the Continuum}, P. Schuster, U. Berger and H. Osswald
(eds.), Synth\`ese Library 306, Kluwer A.P., 2001, 157--166.

\bibitem{11pu1} C. Puritz, Ultrafilters and standard functions in
nonstandard analysis, \emph{Proc. London Math. Soc.} \textbf{22}(1971),
706--733.

\bibitem{11pu2} C. Puritz, Skies, constellations and monads, in \emph{%
Contributions to Non-Standard Analysis}, W.A.J. Luxemburg and A. Robinson
(eds.), North Holland, 1972, 215--243.
\end{thebibliography}
\end{document}